\documentclass[12pt,a4paper]{article}

\usepackage{amsfonts}
\usepackage{amssymb}
\usepackage{graphicx,color}
\usepackage{xcolor}
\usepackage{epstopdf}
\usepackage{epsfig}

\usepackage{amsmath}

\usepackage{amsthm}

\usepackage[bookmarksopen, naturalnames]{hyperref}
\numberwithin{equation}{section}

\theoremstyle{plain}
\newtheorem{theorem}{Theorem}[section]
\newtheorem{corollary}[theorem]{Corollary}
\newtheorem{proposition}[theorem]{Proposition}
\newtheorem{lemma}[theorem]{Lemma}
\theoremstyle{definition}
\newtheorem{definition}[theorem]{Definition}

\theoremstyle{remark}
\newtheorem{remark}[theorem]{Remark}

\newcommand{\cp}{\mathop{\rm cap}}

\newcommand{\field}[1]{\mathbb{#1}}

\newcommand{\R}{\field{R}}

\newcommand{\N}{\field{N}}

\newcommand{\PP}{{\mathcal P}}

\renewcommand{\bar}{\overline}

\renewcommand{\hat}{\widehat}

\newcommand{\eps}{\epsilon}

\def \diam{{\rm diam}}

\def\XXint#1#2#3{{\setbox0=\hbox{$#1{#2#3}{\int}$}
\vcenter{\hbox{$#2#3$}}\kern-.5\wd0}}

\begin{document}
\title{
Riesz energy problems with external fields \\
and related theory}
\author{P. Dragnev, R. Orive, E. B. Saff and F. Wielonsky}
\date{\today}
\maketitle
\begin{center}
Dedicated to Natalia Zorii for her extensive contributions to potential theory
\end{center}
\begin{abstract}
In this paper, we investigate Riesz energy problems on 
unbounded conductors in $\R^d$ in the presence of general external fields $Q$, not necessarily satisfying the growth condition $Q(x)\to+\infty$ as $|x|\to+\infty$ assumed in several previous studies.
We provide sufficient conditions on $Q$ for the existence of an equilibrium measure and
the compactness of its support. Particular attention is paid to the case of the hyperplanar conductor $\R^{d}$, embedded in $\R^{d+1}$, when the external field is created by the potential of a signed measure $\nu$ outside of $\R^{d}$.
Simple cases where $\nu$ is a discrete measure are analyzed in detail. New theoretic results for Riesz potentials, in particular an extension of a classical theorem by de La Vall\'ee-Poussin, are established. These results are of independent interest.
\end{abstract}
\\[\baselineskip]
\section{Introduction}

This paper is devoted to the study of weighted Riesz $s$-equilibrium measures in $\R^d$, $d\geq2$.
For $0<s<d$, we will be dealing with their associated Riesz $s$-potentials \textcolor{black}{and Riesz $s$-energies} of the form
\begin{equation}\label{defpot}
U^{\sigma}(x) =U_{s}^\sigma (x):= \int \frac{d\sigma (t)}{|x-t|^s}, \qquad 
\textcolor{black}{I(\sigma)=I_{s} (\sigma):= \iint\frac{d\sigma (x)d\sigma (t)}{|x-t|^s},}
\end{equation}
for \textcolor{black}{(signed)} measures $\sigma$ supported in $\R^d$, where $|x|$ denotes the euclidean norm of $x$ in $\R^{d}$. Potentials as in (\ref{defpot}) will  also serve to define our external fields, this time with measures $\sigma$ supported in the ambient space $\R^{d+1}$.

Note that
$$U_0^{\sigma}(x) := - \int \log |x-t|\,d\sigma (t)\,,$$
is the limit case of the Riesz potentials as $s\rightarrow 0^+$ (see e.g. \cite{Land}); hereafter, we refer to this case as the log case or, simply, as $s=0$. In the present paper we
assume
$0<s<d$, and in most cases we restrict ourselves to $d-2 \leq s < d$.

In the case of logarithmic potentials it is common to study energy problems on \textcolor{black}{a closed unbounded conductor} $\Sigma$ in the complex plane in the presence of an external field $Q$ defined on $\Sigma$. It is well known that if $Q$ is admissible, that is, \textcolor{black}{lower semicontinuous} such that the logarithmic capacity of $\{z\in \Sigma, \,Q(z) < \textcolor{black}{+\infty}\}$ is positive,
and $Q$ satisfies the condition
\begin{equation}\label{logadmiss}
\lim_{|z| \rightarrow \textcolor{black}{+\infty},\,z\in\Sigma}(Q(z) - \log |z|) = +\infty,
\end{equation}
then the existence of the equilibrium measure and the compactness of its support hold true,
see \cite[Theorem I.1.3]{ST}.
Furthermore, in the last 15 years a class of external fields satisfying a weaker growth condition than \eqref{logadmiss}, namely,
\begin{equation*}
\lim_{|z| \rightarrow \textcolor{black}{+\infty},\,z\in\Sigma}(Q(z) - \log |z|) = M \in (-\infty,+\infty],
\end{equation*}
has received a growing interest. Under this condition it can be proven that there still exists an equilibrium measure, though the compactness of its support is, in general, no longer valid (see \cite{BLW}, \cite{Hardy}, \cite{Hardy-Kuijlaars}, and \cite{Simeonov}, among others); however, recently a broad class of ``weakly admissible'' external fields with compactly supported equilibrium measures was found in \cite{OSW}.

The study
of Riesz $s$-potentials was initiated in \cite{R1,R2}. Related minimum energy problems (or Gauss variational problems) were investigated in \textcolor{black}{\cite{Fu},} \cite{O}, \cite{Land}, \cite{M}, and numerous contributions by N. Zorii, see the discussion before Theorem \ref{equiv-Min-Fro} for precise references. Recent applications of the \textcolor{black}{Riesz potential theory} often involve compact conductors such as balls, spheres or other compact manifolds in $\R^d$, but there are applications dealing with the case of unbounded conductors, in particular, the whole space $\R^d$, see e.g. \cite{chafai}, \cite{leble} and \cite{abey}. In these contributions the condition that
\begin{equation}\label{Q-infty}
Q(x) \to +\infty\quad\text{as}\quad |x|\to+\infty,
\end{equation}
is required to ensure the existence of the equilibrium measure $\mu_Q$ and the compactness of its support; see also Section 4.4 in the recent monograph \cite{BHS}. As pointed out in \cite{abey}, if condition (\ref{Q-infty})
holds, it is easy to adapt the proof of \cite[Theorem I.1.3]{ST} to the case of Riesz potentials for general $0<s<d$ (see also \cite[Theorem 4.4.7]{BHS}). However, this sufficient condition is clearly not sharp, as noted in the also recent article \cite{BDO}, where a simple family of external fields on the real axis for which $\lim_{|x| \rightarrow +\infty} Q(x) = 0$ is shown to have compactly supported equilibrium measures.

In the present paper we aim to study Riesz $s$-energy problems in $\R^d$ for
general external fields $Q$ \textcolor{black}{that fail to satisfy} (\ref{Q-infty}). In particular, the case of the hyperplanar conductor $\R^{d}$ in the presence of external fields of the form $Q(x) = U^{\nu} (x)$, where $\nu$ is a signed measure compactly supported on $\R^{d+1}\setminus \R^{d}$, is analyzed in detail. Note that here
we shall consider $\R^{d}$ as embedded in $\R^{d+1}$; that is, $\R^{d} = \R^{d} \times \{0\} \subset \R^{d+1}$.

The outline of the paper is as follows. In Section \ref{main} we state and comment on our main results, first regarding general external fields $Q$ and second, relating to the particular case where $Q$ equals the potential $U^{\nu}$, with $\nu$ as above. Section \ref{preliminary} is devoted to some definitions, properties and auxiliary results useful for the proofs of the main results.
In particular we prove a version for Riesz potentials of a classical result by de La Vall\'ee Poussin for Newton potentials. We also show the monotonicity of the Robin constant with respect to the external field.
These results may be of independent interest.
In Section \ref{examples} the simplest cases corresponding to discrete measures $\nu$ are analyzed with more detail. Finally, Section \ref{proofs} is devoted to the proofs of our theorems.
\section{Riesz energy problems in $\R^d$ and main results}
\label{main}
Hereafter, all the positive measures $\mu$ considered are finite on compact sets and, when they have unbounded supports, they
 integrate the Riesz $s$-kernel at infinity, namely
\begin{equation}\label{cond-inf}
\int_{|t|>1}\frac{d\mu(t)}{|t|^{s}}<+\infty.
\end{equation}
Note, that any probability measure trivially satisfies (\ref{cond-inf}).

Let
\begin{itemize}
\item $\Sigma$ be a closed (possibly) unbounded \textcolor{black}{conductor in} $\R^d$,
\item $\PP(\Sigma)$ be the set of probability measures on $\Sigma$,
\item $Q(x)$ be a \textcolor{black}{lower semicontinuous} function on $\Sigma$, lower bounded if $\Sigma$ is unbounded, such that $\{x\in\Sigma,~Q(x)<+\infty\}$ has positive $s$-capacity (see below for the definition of capacity).
\end{itemize}

Now, we consider the following energy problem for Riesz $s$-potentials: find
a measure $\mu_Q\in\PP(\Sigma)$ that minimizes the weighted energy
\begin{equation}\label{weightenergy}
I_Q(\sigma) := \textcolor{black}{I(\sigma)} + 2\int Q(x)d\sigma (x) ,
\end{equation}
among all measures $\sigma$ in $\PP(\Sigma)$. We denote by $W_{Q}(\Sigma)$ this infimum. \textcolor{black}{When $Q=0$, the infimum of $I (\sigma)$ is called the {\em $s$-energy} of the set $\Sigma$ and is denoted by $W(\Sigma)$}. When $\Sigma=K$ is compact, its $s$-energy is finite and we define its {\em $s$-capacity} \footnote{\textcolor{black}{Alternative definition of a capacity of a compact set $K$  due to de La Vall\'{e}e-Poussin is given as $\cp(K)=\max \mu(K)$, where the maximum is taken over positive measures supported on $K$ \textcolor{black}{such that $U^\mu (x)\leq 1$ on $S_\mu$, the support of $\mu$ }(see \cite[p. 139]{Land}).} } as the \textcolor{black}{reciprocal of its $s$-energy, namely $\cp(K) := W(K)^{-1}$,} and for a Borel set $B$,
$$
\cp(B):=\sup\{\cp(K),~K\subset B,~K\text{ compact}\}.
$$
Every set that is contained in a Borel set of zero capacity is considered to have capacity zero.
As usual, we say that an inequality holds quasi-everywhere (abbreviated q.e.) on a set if the set of exceptional points is of capacity zero or, equivalently, has infinite energy.

Our first Theorem is a reminder of basic properties about the Riesz energy problem. These properties are well-known in the case of the logarithmic kernel.
For the Riesz kernel, 
they are particular cases of more general results, obtained by N. Zorii, about energy problems for families of noncompact conductors on a locally compact space, with general kernels and external fields (also in the form of a potential of a signed measure of finite energy), see \cite[Theorems 1-4]{Z0}, \cite[Theorem \textcolor{black}{1 and 2 and Proposition 1}]{Z1}, \cite[Theorems \textcolor{black}{7.1-7.3}]{Z2}. The results stated in the next theorem can also be found in \cite[Theorem 1.2]{chafai} or \cite[Theorem 4.4.14]{BHS}, when (\ref{Q-infty}) is satisfied.
One of the main properties is a characterization of a minimizing measure $\mu_{Q}$ (also called equilibrium measure), whenever it exists, in terms of
the so-called Frostman inequalities. Throughout, we will denote by $S_{\mu}$ the support of a measure $\mu$.
\begin{theorem}\label{equiv-Min-Fro}
Assume $0<s<d$. Then, \\
(i) $-\infty<W_{Q}(\Sigma)<\textcolor{black}{+\infty}$;\\
(ii) if a minimizing measure $\mu_{Q}$ exists, it is unique; \\
(iii) if $\mu_{Q}$ is a minimizing measure, the Frostman inequalities
\begin{align}
U^{\mu_Q} (x) + Q(x)  & \geq F_{Q},\quad \textcolor{black}{\text{q.e. on}}\; \Sigma,\label{Frostman1}\\
U^{\mu_Q} (x) + Q(x)  & \leq F_{Q},\quad x\in S_{\mu_Q}, \label{Frostman2}
\end{align}
hold true, where the finite constant
$$F_{Q}=I(\mu_{Q})+\int Qd\mu_{Q},$$
is referred to as the equilibrium constant or the (modified) Robin constant. \\
(iv) Conversely, if $\mu\in\PP(\Sigma)$ is a measure with $I_{Q}(\mu)<+\infty$ satisfying
\begin{align}
U^{\mu} (x) + Q(x)  & \geq F,\quad \textcolor{black}{\text{q.e. on}}\; \Sigma,\label{Frostman3}
\\
U^{\mu} (x) + Q(x)  & \leq F,\quad \textcolor{black}{\text{q.e. on}}\; S_{\mu}, \label{Frostman4}
\end{align}
with some \textcolor{black}{finite} constant $F$, then $\mu$ is a minimizing measure and $F=F_{Q}$. Moreover,
if $\mu_{Q}$ exists,
(\ref{Frostman3}) may be replaced with the weaker inequality
\begin{equation}\label{Frostman5}
U^{\mu} (x) + Q(x) \geq F,\quad \textcolor{black}{\text{q.e. on}}\; S_{\mu_Q}.
\end{equation}
\end{theorem}

For the sake of completeness, we will provide a proof of Theorem \ref{equiv-Min-Fro}, see Section \ref{proofs}.
\textcolor{black}{\begin{remark} We note that when $\lim_{|x|\to +\infty, x\in \Sigma} Q(x)=+\infty$ Theorem \ref{equiv-Min-Fro} reduces to the case of a compact $\Sigma$, because the minimization of \eqref{weightenergy} over $\mathcal{P}(\Sigma)$ is equivalent to a minimization over $\mathcal{P}(\Sigma\cap \{x:Q(x)\leq M \})$ for some finite $M>0$ (see e.g. \cite[Theorem 4.4.14.v]{BHS}). Therefore, there is no loss of generality when we restrict to a finite limit in the assumption A2 below. 
\end{remark}}

For the next results, we will need the notion of thinness of a set at infinity, due, in the \textcolor{black}{Newton} case, to M. Brelot  \cite{Br} and J.L. Doob \cite[pp.\ 175-176]{D}. This notion has been also defined in the Riesz case ($0<s<d$) by Kurokawa and Mizuta in \cite{KM}, and since we will need one result of \cite{KM}, see Lemma \ref{Mizu} below, we will stick to the definition given there, namely that a Borel set $\Sigma$ is thin at infinity if it satisfies the following Wiener condition :
$$
\sum_{k=1}^{+\infty} 2^{-ks} \cp\left(\Sigma^{(k)}\right)<+\infty,
$$
where 
$\Sigma^{(k)}=\left\{x \in \Sigma,~ 2^{k} \leq|x|<2^{k+1}\right\}$. Denoting by
$\Phi$ the inversion $\Phi(x)=x/|x|^{2}$ with respect to the unit sphere of $\R^{d}$, the previous inequality is equivalent to (see \cite[p.287]{Land}),
$$
\sum_{k=1}^{+\infty} 2^{ks} \cp\left(\Phi(\Sigma^{(k)})\right)<+\infty.
$$
When $d-2\leq s<d$, this is a Wiener necessary and sufficient condition for the point 0 to be an irregular point of
\textcolor{black}{the complement of} $\Phi(\Sigma)$. It is known from \cite[Theorem 5.10]{Land} that 
\textcolor{black}{a point $x$ is an irregular point of the complement of a set $E$ if and only if $E$ is thin at $x$ (here, we follow the terminology of most authors, see e.g. \cite{M, Ra, ST}, while, in the terminology of \cite{Land}, $x$ is an irregular point of $E$)}. Hence, with the definition above and when $d-2\leq s<d$, the set $\Sigma$ is thin at infinity if and only if $\Phi(\Sigma)$ is thin at 0.
\textcolor{black}{\textcolor{black}{Following \cite[p. 78]{M}}, when $0<s<d$, we define a Borel set $\Sigma$ to be thin at a $x_{0}\in\R^{d}$ if 
$$
\sum_{k=1}^{+\infty} 2^{ks} \cp\left(\Sigma_{x_{0}}^{(k)}\right)<+\infty,
$$
where 
$\Sigma_{x_{0}}^{(k)}:=\left\{x \in \Sigma,~ 2^{-k} \leq|x-x_{0}|<2^{-k+1}\right\}$.
}

For more details about the notions of thinness of a set at a point and the fine topology, see e.g. \cite[Chapter V, \S3]{Land}. For a recent study of thinness of a set at infinity in the theory of Riesz potentials, in particular how the notions introduced by Brelot and Doob extend to this framework, see \cite{Z3}.

In the sequel, {\textcolor{black}{unless explicitly stated otherwise, whenever $\Sigma$ is unbounded, we assume:}
\\[.5\baselineskip]
\textbf{A1)} 
$\Sigma$ is not thin at infinity.
\\[.5\baselineskip]
\textbf{A2)} $Q$ can continuously be extended to the point at infinity\footnote{\textcolor{black}{To avoid ambiguity, we denote by $\infty$ the Alexandroff point of ${\mathbb R}^d$ and use the explicit notations of $-\infty$ and $+\infty$ in the context of the real line.}}; that is, there exists the limit, \textcolor{black}{
$$\lim_{|x|\to+\infty,~x\in \Sigma}\,Q(x) = : Q(\infty),$$
and moreover this limit is finite:
$$-\infty <Q(\infty)<+\infty .$$}
\begin{remark}\label{Rem-Frost}
\textcolor{black}{We claim that  when $\Sigma$ is unbounded and $\mu_Q$ exists, the assumptions A1-A2 and \eqref{Frostman1} imply $Q(\infty)\geq F_{Q}$. Indeed, by Lemma \ref{Mizu} below, there is a Borel set $P\subset{\mathbb{R}^d}$, thin at infinity, and such that 
\begin{equation}\label{Z6}\lim_{|x|\to +\infty, \, x\not\in P} U^{\mu_Q} (x) =0.
\end{equation} 
As $\Sigma$ is not thin at infinity, we have $\cp(\Sigma\setminus P)>0$, and hence there is a sequence of points $x_{n} \in \Sigma\setminus P$, $|x_n|\to +\infty$, for which \textcolor{black}{both \eqref{Frostman1} and  \eqref{Z6} hold true.}}
\end{remark}

Our next result concerns the existence of a minimizing measure on an unbounded set $\Sigma$, in a general external field $Q$, depending on its behavior at infinity. In that connection, see also \cite[Theorem 8.1, Proposition 8.1]{Z2}.
\begin{theorem}\label{thm:general}
Let $\Sigma$ be an unbounded \textcolor{black}{conductor in} $\R^{d}$ and $Q$ an external field satisfying conditions A1--A2. Then, for $0<s<d$,
\begin{itemize}
\item[(i)] sufficient conditions for the existence of a solution $\mu_Q$ to the Riesz energy problem
are the following growth conditions on $Q$ at infinity:
\begin{equation}\label{cond-Q}
|x|^{s}(Q(x) - Q(\infty))\leq  - 1,\qquad\text{for large $|x|$ },\ x\in\Sigma;
\end{equation}
or
\begin{equation}\label{cond2-Q}
\lim_{|x|\to+\infty,~x\in\Sigma\setminus P} |x|^{s}(Q(x)- Q(\infty))
\leq-1,
\end{equation}
for some set $P$ thin at infinity.
\item[(ii)] \textcolor{black}{
Assume
\begin{equation}\label{beh-Q}
\lim_{|x|\to+\infty,~x\in\Sigma}\,|x|^s\,\left(Q(x) - Q(\infty)\right)
\end{equation}
exists. 
If $\Sigma$ admits a minimizing measure $\mu_{Q}$ with
an unbounded support then
the limit in (\ref{beh-Q}) equals $-1$.
}
\item[(iii)] if $d-2\leq s<d$ and there exists $c<1$ such that
\begin{equation}\label{cond-Q2}
-\frac{c}{|x|^{s}}\leq Q(x)-Q(\infty),\qquad x\in \Sigma\,, 
\end{equation}
then the Riesz energy problem has no solution.
\end{itemize}
\end{theorem}
\begin{remark}
 \textcolor{black}{Actually, it can be shown by a similar argument that the conclusion of Theorem \ref{thm:general}(iii) holds  more generally if, instead of \eqref{cond-Q2}, we assume that 
\begin{equation}
-{c}U^\mu (x)\leq Q(x)-Q(\infty), \quad x\in \Sigma,
\end{equation}
where $c<1$ and $\mu\in \mathcal{P}(\Sigma)$ has compact support.} We note also that when $Q$ equals a constant  on an unbounded conductor and $d-2\leq s<d$, the Riesz energy problem has no solution, for condition \eqref{cond-Q2} is satisfied.\end{remark}

Theorem \ref{thm:general} applies to the particular case where the external field $Q$ is due to the action of an attractive mass placed at a point outside of the conductor $\Sigma$
(so that $Q$ is \textcolor{black}{continuous} on $\Sigma$).
For instance, assuming $0\notin \Sigma$,
 we have the following.
\begin{corollary}\label{cor:deltaorigin}
Under the assumptions in Theorem \ref{thm:general}, if $0\notin \Sigma$ and for some $c>0$,
\begin{equation*}
Q(x)=-\frac{c}{|x|^{s}},\qquad x\in\Sigma\,,
\end{equation*}
then\\
(i) for $0<s<d$ and $c\geq1$, the minimizing measure $\mu_{Q}$ exists and, for $c>1$, it has bounded support;\\
(ii) for $d-2\leq s<d$ and $c<1$, a minimizing measure $\mu_{Q}$ does not exist.
\end{corollary}
\begin{remark}
Corollary \ref{cor:deltaorigin} extends for a general dimension $d\geq 2$ the results obtained in \cite[Theorem 2.1]{BDO} for the case where $d=2$, $\Sigma = \R$ and $Q$ is created by a single attractive charge placed at a point in $\R^{2} \setminus \R\,.$
When $c=1$, the support of the weighted equilibrium measure can be unbounded, as, for instance, \textcolor{black}{in \cite[Lemma 3.1]{BDO}}, or more generally, when $\Sigma=\R^{d}$ and $Q$ is created by an attractive charge placed at a point in $\R^{d+1} \setminus \R^{d}$, see Section \ref{w-Q-charges}, case (i). It can also be bounded, see Remark \ref{rem-locat} for more details. 
\end{remark}
In the current paper, we will also consider in detail the case where the conductor $\Sigma = \R^{d}$ and $Q$ is the potential $U^{\nu}$ of a signed measure $\nu$
supported in $\R^{d+1}\setminus \R^{d}$,
$$\textcolor{black}{Q(x) = U^{\nu}(x)}.$$

An auxiliary result about Riesz potentials will be useful.
\begin{lemma}[{cf.\ \cite[Theorem 3.3]{KM}}]\label{Mizu}
Let $\mu$ be a positive measure in $\R^{d}$. Then, there exists a Borel set $P$ thin at infinity such that
\begin{align}
\lim_{|x|\to+\infty,~x\in\R^{d}\setminus P}U^{\mu}(x) & =0, \label{Mizu1}
\\[10pt]
\lim_{|x|\to+\infty,~x\in\R^{d}\setminus P}|x|^{s}U^{\mu}(x) & =\mu(\R^{d}),\qquad  \text{if }\mu(\R^{d})<\infty. \label{Mizu2}
\end{align}
\end{lemma}
From this lemma, we see that condition (\ref{cond2-Q}) of assertion (i) and assertion (ii) of Theorem \ref{thm:general} imply the following.
\begin{corollary}\label{cor:admiss}
Let $0<s<d$. Let $\nu$ be a signed measure of finite mass (both, $\nu^+$ and $\nu^-$ in the Jordan decomposition  $\nu=\nu^+-\nu^-$ have finite mass) with support in $\R^{d+1}\setminus \R^{d}$ (not necessarily compact), and let \textcolor{black}{$\Sigma=\mathbb{R}^d$ and} $Q = U^{\nu}$ 
be as above.
If $\nu (\R^{d+1}) \leq-1$, then there exists a unique equilibrium measure $\mu_Q$. Moreover, if
$\nu (\R^{d+1}) <-1$, the support $S_{\mu_Q}$ is a compact subset of $\R^{d}$.
\end{corollary}
Indeed, we have $Q=U^{\nu}=U^{ \nu^w}$ and $\nu(\R^{d+1})=\nu^w(\R^{d})$, where $\nu^w$ denotes the weak balayage  of $\nu$ onto $\R^{d}$ (see Section \ref{preliminary} for the definition and Lemma \ref{lem:balayage} for details). Moreover (\ref{Mizu2}) can be applied to each terms in the Jordan decomposition of the signed measure $\nu^w=\nu_{+}^w-\nu_{-}^w$, and the union of two sets thin at infinity is also thin at infinity.

The situation where \textcolor{black}{$\Sigma$}, $Q$ and $\nu$ are as above,
and $\nu (\R^{d+1})= -1$, will be referred to as a {\it weakly admissible setting}, using the parallelism with the logarithmic potential scenario.
For this purpose, we denote by $y = (y_1,\ldots,y_d,y_{d+1})$ an arbitrary point in $\R^{d+1}$
so that the distance from $y$ to the conductor $\R^{d}$ equals $|y_{d+1}|$.

Our final main result is the following one.
\begin{theorem}\label{thm:weakadmiss}
Let $d-2\leq s<d$, \textcolor{black}{$\Sigma=\mathbb{R}^d$}, and $Q = U^{\nu}$, $\nu (\R^{d+1}) = -1$. If
\begin{equation}\label{sufficRd}
\int\,|y_{d+1}|^{d-s}d\nu(y) > 0\,,
\end{equation}
then $S_{\mu_Q}$ is a compact subset of $\R^{d}$.
\end{theorem}
\begin{remark}\label{rem:weakadmiss}
Theorem \ref{thm:weakadmiss} extends to Riesz potentials what was established in \cite{OSW} for the logarithmic kernel in $\R^{2}$ with $\Sigma = \R$ in the weakly admissible setting. Here, when the total attractive charge of the external field equals in size the charge to be spread on $\R^{d}\,$ it is possible to get a compactly supported equilibrium measure. However, in this border case not just the sizes of the charges are important but also their distances to the conductor. This very interesting phenomenon may be illustrated with a simple example. Namely, if the external field consists just of an attractor of unit charge, then the equilibrium measure exists but its support will be unbounded (the whole conductor $\R^{d}$ in the current setting); but if this simple configuration is slightly modified by adding a negative charge $- \varepsilon$ to this attractor, \textcolor{black}{and placing a repellent of positive charge \textcolor{black}{of equal magnitude} at another point outside \textcolor{black}{the} conductor (observe that the total attractive charge remains equal to one), it is possible to place that repellent sufficiently far from the conductor so that the equilibrium measure becomes compactly supported.}
\end{remark}

With the level of generality considered so far, not much more can be said about the equilibrium measure $\mu_Q$ and its support $S_{\mu_Q}$. In Section 4 we will discuss some particular situations where the simplicity of the external field allows for a more detailed analysis.
\section{Auxiliary results}
\label{preliminary}
In this section properties and auxiliary results related to Riesz energy problem are proved. Also, we recall the notions of Kelvin transform, balayage and signed equilibrium measures, which will be useful in the sequel. Throughout, it will also be convenient to introduce the parameter $\alpha$, $0<\alpha<d$, given by
$$
\alpha:=d-s.
$$
\subsection{A de La Vall\'ee Poussin theorem for Riesz kernels}
In this subsection, we give a version for Riesz potentials of the following result by
de La Vall\'ee--Poussin for \textcolor{black}{Newton} potentials, or more generally, for superharmonic functions, see \cite[p.\ 21]{DLVP}:
\begin{theorem}\label{dlvp-class}
Let $u$ and $v$ be two superharmonic functions defined in an open subset $\Omega$ of $\R^{d}$, with respective Riesz measures \textcolor{black}{$\mu=-\Delta u$ and $\nu=-\Delta v$}. Assume $u\geq v$ in $\Omega$. Then the restriction of the signed measure $\mu-\nu$ to the set
$$E=\{x\in\Omega,~
u(x)=v(x)<\infty\}$$
is a negative measure.
\end{theorem}
See also \cite[Theorem IV.4.5]{ST} for the logarithmic case in the complex plane. A similar result was proved by Janssen \cite[Theorem 2.5]{J} in a general potential theoretic setting, see also Fuglede \cite[Theorem 1.1]{F} for \textcolor{black}{a refinement of Theorem \ref{dlvp-class} to $\delta$-subharmonic functions}.

In this subsection, \textcolor{black}{unless explicitly stated otherwise,} it is assumed that $d-2<s<d$ or equivalently $0<\alpha<2$.

We can now state the main result of this subsection.
\begin{theorem}\label{dlvp}
Let $\mu$ and $\nu$ be two positive measures such that the potential \textcolor{black}{$U^{\nu}$ is
finite $\nu$-a.e. (for instance $\nu$ has finite energy)}.
If for some constant $C\geq0$,
\begin{equation}\label{ineq-pot}
U^{\nu}\leq U^{\mu}+C,\quad\text{$\nu$-almost everywhere},
\end{equation}
then the restriction of the signed measure $\mu-\nu$ to the set
$$E=\{x\in\R^{d},~
U^{\nu}(x)=U^{\mu}(x)+C<\infty\}$$
is a negative measure.
\end{theorem}
\begin{remark}
Notice that the statement of Theorem \ref{dlvp-class}
does not hold for Riesz subharmonic kernels $|x|^{-s}$ with $d-2<s<d$. Indeed, considering the equilibrium measure $\omega$ of the unit ball $B\subset\R^{d}$, we know that
$$
U^{\omega}(x)=W(B),\qquad x\in B,
$$
where $W(B)$ denotes the Riesz energy of $B$. Applying Theorem \ref{dlvp-class}
with $U^{\omega}$ and the constant function equal to $W(B)$, the inequality $W(B)\leq U^{\omega}(x)$
would imply that $\omega\leq0$ on any open subset $\Omega$ of $B$, which is not true as the support of $\omega$ is the entire ball $B$ when $d-2<s<d$.
\end{remark}
Let us first recall the notion of {\em balayage}, see \cite[Chapter IV, \S 5]{Land}, in particular p.264. Given
a closed set $F\subset \R^d$ of positive capacity and a \textcolor{black}{positive} measure $\sigma$ \textcolor{black}{of finite total mass}, 
there exists a unique \textcolor{black}{positive} measure
$$\widehat{\sigma}:=
Bal(\sigma,F)$$
called the Riesz $s$-balayage of $\sigma$ onto $F$ satisfying the following: \\
$\bullet$ $S_{\widehat{\sigma}} \subseteq F$, \\
$\bullet$ $\widehat{\sigma}$ is zero on the set of irregular points of \textcolor{black}{the complement of $F$}, and
\begin{equation}\label{balayage}
U^{\widehat{\sigma}}(x) = U^{\sigma}(x)\text{ q.e.\ on }F,\qquad U^{\widehat{\sigma}}(x) \leq U^{\sigma}(x)\text{ on }\R^d.
\end{equation}
\textcolor{black}{We caution the reader that these properties are different from those for the logarithmic case when $d=2$ and $s=0$. In particular,} the \textcolor{black}{logarithmic} balayage preserves the total mass, but this is not true for general $d-2 \leq s < d$.
\textcolor{black}{For positive $s$,} $\|\widehat{\sigma}\| = \|\sigma\|$ for all measures $\sigma$ if and only if $F$ is not thin at infinity, see \cite[Theorem 3.22]{FZ}, \cite[Corollary 5.3]{Z3}.
%

We remark that for $0<s<d-2$ we shall utilize a measure $\sigma^w$, which we call {\em weak balayage}, for which only the first part of \eqref{balayage} holds true, namely $U^{\sigma^w}(x)=U^\sigma (x)$ q.e. on $F$. For a signed measure $\nu$, we define $\nu^w :=(\nu^+)^w-(\nu^-)^w$ , where $\nu=\nu^+-\nu^-$ is the Jordan decomposition of $\nu$.

Now we consider the family of balayages $\hat\delta_{0,r}$, $r>0$, of $\delta_{0}$ (the unit point mass at $x=0$) outside of the balls $|x|<r$. The balayage $\hat\delta_{0,r}$ is a probability measure, \textcolor{black}{absolutely continuous with respect to the Lebesgue measure, 
see \cite[p.265]{Land}},
supported on the complement of $B(0,r)$, \textcolor{black}{the open ball with center at $0$ and radius $r$}. Its
potential satisfies 
$$
U^{\hat\delta_{0,r}}(x)=
\begin{cases}
U^{\delta_{0}}(x)={|x|^{-s}},\quad |x|\geq r,\\[10pt]
U^{\hat\delta_{0,r}}(x)\leq{|x|^{-s}},\quad |x|\leq r.
\end{cases}
$$
For $x_{0}\in\R^{d}$, we set $\hat\delta_{x_0,r}(x):=\hat\delta_{0,r}(x-x_{0})$ and, for
$u$ a potential, we will consider its average with respect to $\hat\delta_{x_0,r}$:
$$
L_{u}(x_{0},r)=\int u(x)d\hat\delta_{x_0,r}(x),\qquad r>0.
$$
Note that, when \textcolor{black}{$u(x)=U^{\mu}(x)$} is the potential of a \textcolor{black}{(positive)} measure $\mu$, 
$L_{u}(x_{0},r)$ is a finite number. Indeed,
$$
L_{u}(x_{0},r)=\int u(x)d\hat\delta_{x_0,r}(x)=\int U^{\hat\delta_{x_0,r}}(x)d\mu(x),
$$
and for $0<r_{1}<r$, we have
\begin{equation*}
\int_{r_{1} \leq|x-x_{0}|} U^{\hat\delta_{x_0,r}}(x)d\mu(x)  \leq
\int_{r_{1}\leq|x-x_{0}|} \frac{d\mu(x)}{|x-x_{0}|^{s}}<\infty,
\end{equation*}
where the last inequality follows from (\ref{cond-inf}), and
\begin{equation*}
\int_{|x-x_{0}|<r_{1}} U^{\hat\delta_{x_0,r}}(x)d\mu(x)
\leq
\int_{|x-x_{0}|<r_{1}} \frac{d\mu(x)}{|r-r_{1}|^{s}}<\infty,
\end{equation*}
where the first inequality uses the fact that the mass of $\hat\delta_{x_0,r}$ is outside of the ball $B(x_{0},r)$.

For the proof of Theorem \ref{dlvp}, we will need several lemmas.
\begin{lemma}
Let  $x_{0}\in\R^{d}$ and 
$$
\textcolor{black}{u(x):=U^{\mu}(x)},
$$
with $\mu$ a positive measure. The function
$r\to L_{u}(x_{0},r)$ satisfies
\begin{equation}\label{lim-L}
\lim_{r\to0}L_{u}(x_{0},r)= u(x_{0}),
\end{equation}
and is non-increasing with $r$.

Moreover, the function $r\mapsto L_{u}(x_{0},r)$ is absolutely continuous on any closed subinterval of $(0,\infty)$ and, almost everywhere, one has
\begin{equation}\label{exp-derL}
L_{u}'(x_{0},r)=-\frac{2c_{d,\alpha}}{r^{d-1}}\int_{|x-x_{0}|\leq r}(r^{2}-|x-x_{0}|^{2})^{\alpha/2-1}d\mu(x),
\end{equation}
where
\begin{equation}\label{const-c}
c_{d,\alpha}=\frac{\Gamma(d/2)}{2^{\alpha}\pi^{d/2}\Gamma^{2}(\alpha/2)}.
\end{equation}
\end{lemma}
\begin{remark}
Note that (\ref{exp-derL}) is finite for almost every $r>0$. Note also that in the \textcolor{black}{Newton} case $\alpha=2$, (\ref{exp-derL})
simplifies to
\begin{equation}\label{rel-newton}
-r^{d-1}L_{u}'(x_{0},r) = 2c_{d,2}{\mu(B(x_{0},r))}.
\end{equation}
In view of \textcolor{black}{\eqref{lim-L}} we will define $L_{u}(x_{0},0)$ as $u(x_{0})$.
\end{remark}
\begin{proof}
Equality (\ref{lim-L}) and the fact that $r\mapsto L_{u}(x_{0},r)$ is non-increasing are shown in \cite{Land} p.\ 114 and p.\ 126 respectively.

In the sequel, we choose $x_{0}=0$, the proof being identical at any other point $x_{0}\in\R^{d}$.
Let $r_{1}$ be a fixed positive number with $0<r_{1}<r$. We write
\begin{align}
L(r) -L(r_{1})& :=L_{u}(0,r)-L_{u}(0,r_{1})\notag
\\[10pt]
& = \int_{|x|\leq r_{1}}(U^{\hat\delta_{0,r}}(x)-U^{\hat\delta_{0,r_{1}}}(x))d\mu(x)+\int_{r_{1}<|x|\leq r}(U^{\hat\delta_{0,r}}(x)-U^{\delta_{0}}(x))d\mu(x).\label{exp-L}
\end{align}
To reexpress the last integrals, we will use the following expression for the Green function of the ball of radius $r$, see e.g. \cite[Theorem 3.1]{Buc}
$$
G_{r}(x,y)=U^{\delta_{y}}(x)-U^{\hat\delta_{y,r}}(x)=
\frac{c_{d,\alpha}}{|x-y|^{d-\alpha}}\int_{0}^{a}s^{\alpha/2-1}(1+s)^{-d/2}ds,\qquad
x,y\in B(0,r),
$$
where $c_{d,\alpha}$ is the constant (\ref{const-c}) and
$$a=\frac{(r^{2}-|x|^{2})(r^{2}-|y|^{2})}{r^{2}|x-y|^{2}}.
$$
When $y=0$, we get
\begin{equation*}
G_{r}(x,0)=U^{\delta_{0}}(x)-U^{\hat\delta_{0,r}}(x)=
\frac{c_{d,\alpha}}{|x|^{d-\alpha}}\int_{0}^{a}s^{\alpha/2-1}(1+s)^{-d/2}ds,\qquad
a=\frac{r^{2}-|x|^{2}}{|x|^{2}}.
\end{equation*}
Making use of the above expression, (\ref{exp-L}) becomes
\begin{multline*}
-\int_{|x|\leq r_{1}} \frac{c_{d,\alpha}}{|x|^{d-\alpha}}
\int_{a_{1}}^{a}s^{\alpha/2-1}(1+s)^{-d/2}ds~d\mu(x)
\\
-\int_{r_{1}<|x|\leq r} \frac{c_{d,\alpha}}{|x|^{d-\alpha}}
\int_{0}^{a}s^{\alpha/2-1}(1+s)^{-d/2}ds~d\mu(x).
\end{multline*}
Performing the change of variable $t=|x|\sqrt{s+1}$, we get
\begin{multline*}
L(r) -L(r_{1}) = -2c_{d,\alpha}\int_{|x|\leq r_{1}}\left(\int_{t=r_{1}}^{r} {(t^{2}-|x|^{2})^{\alpha/2-1}}{t^{1-d}}dt\right)d\mu(x)
\\
-2c_{d,\alpha}\int_{r_{1}<|x|\leq r}\left(\int_{t=|x|}^{r} {(t^{2}-|x|^{2})^{\alpha/2-1}}{t^{1-d}}dt\right)d\mu(x).
\end{multline*}
Since the integrand in the double integrals is positive, we may interchange the order of integrations which leads to
$$
L(r) -L(r_{1})= -2c_{d,\alpha}\int_{t=r_{1}}^{r}\left(\int_{|x|=0}^{t}{(t^{2}-|x|^{2})^{\alpha/2-1}}{t^{1-d}}d\mu(x)\right)dt.
$$
Finally, since both $L(r)$ and $L(r_{1})$ are finite numbers, the double integral is also finite which implies, by Fubini's Theorem, that the function
$$
t\mapsto\int_{|x|=0}^{t}{(t^{2}-|x|^{2})^{\alpha/2-1}}{t^{1-d}}d\mu(x)
$$
is integrable. Hence, $L(r)-L(r_{1})$ is the indefinite integral of an integrable function, from which follows that the function $r\mapsto L(r)$ is absolutely continuous, differentiable almost everywhere, with derivative equal to
$$
-2c_{d,\alpha}r^{1-d}\int_{|x|=0}^{r}{(r^{2}-|x|^{2})^{\alpha/2-1}}d\mu(x).
$$
\end{proof}
The following lemma is a consequence of the previous one.
\begin{lemma}
Assume $L_{u}(x_{0},r)$ is differentiable at $r>0$. Then,
\begin{equation}\label{rel-L-mu}
-r^{d}L_{u}'(x_{0},r)=2c_{d,\alpha}\frac{d}{dr}\int_{t=0}^{r}t(r^{2}-t^{2})^{\alpha/2-1}\mu(B(x_{0},t))dt.
\end{equation}
\end{lemma}
\begin{remark}
Note that in the \textcolor{black}{Newton} case $\alpha=2$, (\ref{rel-L-mu}) just gives (\ref{rel-newton}) again.
\end{remark}
\begin{proof}
We have
\begin{align*}
\int_{t=0}^{r}t(r^{2}-t^{2})^{\alpha/2-1}\int_{|x-x_{0}|=0}^{t}d\mu(x)dt
& = \int_{|x-x_{0}|=0}^{r}\int_{t=|x-x_{0}|}^{r}t(r^{2}-t^{2})^{\alpha/2-1}dtd\mu(x)
 \\[10pt]
& =  \frac12\int_{|x-x_{0}|=0}^{r}(r^{2}-|x-x_{0}|^{2})^{\alpha/2}\int_{s=0}^{1}(1-s)^{\alpha/2-1}ds
d\mu(x)
\\[10pt]
& = \frac1\alpha\int_{|x-x_{0}|=0}^{r}(r^{2}-|x-x_{0}|^{2})^{\alpha/2}d\mu(x),
\end{align*}
where the variable $s$ in the second equality is such that $t^{2}-|x-x_{0}|^{2}=s(r^{2}-|x-x_{0}|^{2})$. Hence, the right-hand side of (\ref{rel-L-mu}) is equal to
\begin{align*}
\frac{2c_{d,\alpha}}{\alpha}\frac{d}{dr}\int_{|x-x_{0}|=0}^{r}(r^{2}-|x-x_{0}|^{2})^{\alpha/2}d\mu(x)
& =2c_{d,\alpha}\int_{|x-x_{0}|=0}^{r}r(r^{2}-|x-x_{0}|^{2})^{\alpha/2-1}d\mu(x)
\\[10pt]
& =-{r^{d}}L_{u}'(x_{0},r).
\end{align*}
\end{proof}
Finally we recall a result from measure theory. Its proof uses the Vitali covering theorem, see \cite{Sodin} for details.
\begin{lemma}\label{sodin}
Let $A$ be a Borel set, and let $\nu$ be a signed measure in $\R^{d}$. Suppose that for each $x\in A$ there is a sequence $t_{n}\downarrow0$ with $\nu(B(x,t_{n}))\geq0$. Then $\nu|_{A}$ is a non-negative measure.
\end{lemma}
\begin{proof}[Proof of Theorem \ref{dlvp}]
From the domination principle, see \cite[p.\ 115]{Land}, we know that (\ref{ineq-pot}) actually holds everywhere on $\R^{d}$.
Hence, with $u=U^{\mu}$ and $v=U^{\nu}$, and for $R>0$,
$$L_{v}(x,R)\leq L_{u}(x,R)+C,\qquad x\in\R^{d}.$$
 Consider a point $x_{0}\in E$.
Then $v(x_{0})=u(x_{0})+C<\infty$ and thus
$$
L_{v}(x_{0},R)-v(x_{0})\leq L_{u}(x_{0},R)-u(x_{0}).
$$
Since $r\mapsto L_{v}(x_{0},r)$ is an absolutely continuous function, we have, for any $\eps>0$,
$$
\int_{\eps}^{R}L_{v}'(x_{0},r)dr=L_{v}(x_{0},R)-L_{v}(x_{0},\eps).
$$
From the non-positivity of $L_{v}'(x_{0},r)$ and the monotone convergence theorem, we get by letting $\eps$ tend to 0,
$$
\int_{0}^{R}L_{v}'(x_{0},r)dr=L_{v}(x_{0},R)-{v}(x_{0}),
$$
and the same holds true for the function $L_{u}$.
Consequently,
$$
\int_{0}^{R}L_{v}'(x_{0},r)dr\leq\int_{0}^{R}L_{u}'(x_{0},r)dr.
$$
Next, considering equation (\ref{rel-L-mu}) applied to $u$ and $v$, taking the difference between the two and then the antiderivative, we get
\begin{equation}\label{rel-antider}
-\int_{r=0}^{R}r^{d}L_{u-v}'(x_{0},r)dr=2c_{d,\alpha}\int_{t=0}^{R}t(R^{2}-t^{2})^{\alpha/2-1}(\mu-\nu)(B(x_{0},t))dt,\quad R>0.
\end{equation}
In the sequel, we fix an $R>0$ and show that there exists an $0<r\leq R$ for which the inequality $\mu(B(x_{0},r))\leq\nu(B(x_{0},r))$ holds true. We proceed by contradiction. Assume that, for all $0<r\leq R$, we have $\mu(B(x_{0},r))>\nu(B(x_{0},r))$. Then, the right-hand side of (\ref{rel-antider}) is a positive number. On the other hand, consider a fixed $\delta>0$. For all $\eps>0$ sufficiently small, we have
$$
-\delta\leq\int_{\eps}^{R}L_{u-v}'(x_{0},r)dr.
$$
Then, for some $R'\in(0,R]$,
$$
-\delta\leq\int_{\eps}^{R}L_{u-v}'(x_{0},r)dr=\int_{\eps}^{R}r^{-d}(r^{d}L_{u-v}'(x_{0},r))dr
=\eps^{-d}\int_{\eps}^{R'}r^{d}L_{u-v}'(x_{0},r)dr,
$$
where in the last equality we have used the Bonnet's mean value theorem for integrals, cf.\ \cite[Section 12.6]{B} (note that $r^{-d}$ is a positive decreasing function). Observing that
$$
0=\lim_{\eps\to0}(-\delta\eps^{d})\leq\lim_{\eps\to0}\int_{\eps}^{R'}r^{d}L_{u-v}'(x_{0},r)dr=\int_{0}^{R'}r^{d}L_{u-v}'(x_{0},r)dr,
$$
we get a contradiction since by (\ref{rel-antider}), the above integral
 is a negative number.

Repeating the argument with a smaller value of $R$, we get a sequence $r_{n}\downarrow0$ such that $\mu(B(x_{0},r_{n}))\leq\nu(B(x_{0},r_{n}))$. It then suffices to apply Lemma \ref{sodin} to conclude the proof.
\end{proof}
\subsection{Monotonicity and continuity of the equilibrium constant}
In this section, we always assume that a minimizing measure exists \textcolor{black}{(see sufficient and necessary conditions for existence in Theorem 2.3)}.
We derive monotonicity and continuity of the equilibrium constant with respect to the external field. This result, which may be of independent interest, is an analog of \cite[Corollary I.4.2]{ST} for the logarithmic kernel. It will be used in the proof Theorem \ref{thm:general}.

Let us start with a lemma.
\begin{lemma}
For $d-2\leq s< d$, \textcolor{black}{and $\nu$ a probability measure on $\R^{d}$},
the following inequality holds
\begin{equation}\label{Robin-ineq}
\textcolor{black}{``\displaystyle{\inf_{x\in S_{\mu_Q}}}"(U^{\nu}(x)+Q(x))\leq F_{Q}}, 
\end{equation}
where $\textcolor{black}{``\inf"}$ denotes the essential inf in the sense of \cite[p.\ 43]{ST}, \textcolor{black}{that is, $\textcolor{black}{``\inf_{x\in H}"}h(x)$ denotes the largest number $L$ such that on $H$ the real function $h$ takes values smaller than $L$ only on a set of zero capacity.}
\end{lemma}
\begin{proof} We essentially follow the proof of \cite[Theorem 1.3]{DS}, where (\ref{Robin-ineq}) was proved in the case of an external field $Q$ on the $d$-dimensional sphere $S^{d}\subset\R^{d+1}$ and $d-1\leq s<d$.
Let 
$L$ be some constant such that $L\leq U^{\nu}+Q$, q.e. on $S_{\mu_Q}$ or
$$
U^{\nu}-U^{\mu_{Q}}\geq L-Q-U^{\mu_{Q}},\quad \text{q.e on } S_{\mu_Q}.
$$
From the second Frostman inequality, $\inf_{S_{\mu_Q}}(-U^{\mu_{Q}}-Q)\geq-F_{Q}$; hence
$$
U^{\nu}-U^{\mu_{Q}}\geq L-F_{Q},\quad \text{q.e.\ on } S_{\mu_Q}, \textcolor{black}{\text{ hence } \mu_{Q}\text{-a.e.}}
$$
Assume now 
that $L>F_{Q}$ and consider
any \textcolor{black}{compact subset $A$ of $S_{\mu_Q}$ of nonzero capacity}. If $S_{\mu_Q}$ is bounded one can take $A=S_{\mu_Q}$.
Let $\omega_{A}$ be the equilibrium measure of $A$ and let $\sigma=\cp(A) \omega_{A}=\textcolor{black}{W(A)^{-1}} \omega_{A}$, where
$\cp(A)$ and $W(A)$ respectively denote the Riesz capacity and Riesz energy of $A$. Then $U^{\sigma}\leq1$ on $\R^{d}$ and in particular on $S_{\mu_Q}$.
Thus we have
$$
U^{\mu_{Q}+(L-F_{Q})\sigma}\leq U^{\nu},\quad\text{q.e. on }S_{\mu_Q}.
$$
By the domination principles \cite[Theorems 1.27 and 1.29]{Land}, the inequality is satisfied everywhere, and then, by the principle of positivity of mass, cf.\ \cite[Theorem 3.11]{FZ} or \cite[Theorem 1.7 p.\ 62]{M},
we get $1+(L-F_{Q})\cp(A)\leq1$, which is a contradiction.
\end{proof}
\begin{proposition}\label{monot}
For $d-2\leq s< d$ \textcolor{black}{and a given $\Sigma \subset \R^d$},
the Robin constant $F_{Q}$ is a non-decreasing, continuous function of $Q$. More precisely, let $C,\epsilon \geq0$ be some nonnegative constants \textcolor{black}{and denote by $F_i, i=1,2$ the Robin constants for the external fields $Q_i, i=1,2$, respectively, \textcolor{black}{assuming they exist}}. Then
\begin{equation}\label{Robin-monot1}
C\leq Q_{2}-Q_{1}\text{ on }S_{\mu_{Q_{2}}}\implies
\begin{cases}
C\leq (U^{\mu_{Q_{1}}}(x)-F_{1})-(U^{\mu_{Q_{2}}}(x)-F_{2}),\quad x\in\R^{d},\\[5pt]
C\leq F_{2}-F_{1},
\end{cases}
\end{equation}
and
\begin{equation}\label{Robin-monot2}
|Q_{2}-Q_{1}|\leq\epsilon\text{ on }S_{\mu_{Q_{1}}}\cup S_{\mu_{Q_{2}}}\implies
\begin{cases}
|(U^{\mu_{Q_{1}}}(x)-F_{1})-(U^{\mu_{Q_{2}}}(x)-F_{2})|\leq\epsilon,\quad x\in\R^{d},\\[5pt]
|F_{2}-F_{1}|\leq\epsilon.
\end{cases}
\end{equation}
\begin{proof}
The proof is based on the following identity:
\begin{equation}\label{caract}
U^{\mu_{Q}}(x)-F_Q=\inf\{U^{\nu}(x)-\text{``$\inf_{x\in \Sigma}$''}(U^{\nu}+Q),~\nu\in\PP(\Sigma) \},\quad x\in\R^{d},
\end{equation}
where the essential ``inf'' can also be taken on $S_{\mu_Q}$ instead of $\Sigma$.
The left-hand side is no smaller than the infimum because $\mu_{Q}\in\PP(\Sigma)$ and ``$\inf_{x\in \Sigma}$''$(U^{\mu_{Q}}+Q)=F_Q$. For the converse inequality, using
\eqref{Frostman2}, note that,
$$
U^{\mu_{Q}}(x)+Q(x)-F_Q+\text{``$\inf_{x\in \Sigma}$''}(U^{\nu}+Q)\leq U^{\nu}(x)+Q(x),\quad x\in S_{\mu_Q},
$$
or equivalently
$$
U^{\mu_{Q}}(x)\leq U^{\nu}(x)+(F_Q-\text{``$\inf_{x\in \Sigma}$''}(U^{\nu}+Q)),\quad x\in S_{\mu_Q},
$$
where the constant term $F_{Q}-\text{``$\inf_{x\in \Sigma}$''}(U^{\nu}+Q)$ is non-negative by (\ref{Robin-ineq}). Then, by the domination principle in the preceding proof, the inequality holds everywhere. Thus the left-hand side in (\ref{caract}) is also no-larger than the right-hand side.

Now, if $C\leq Q_{2}-Q_{1}$ on $S_{\mu_{Q_{2}}}$, we have, for a given $\nu\in\PP(\Sigma)$,
$$
\text{``$\inf_{x\in \Sigma}$''}(U^{\nu}+Q_{1})\leq \text{``$\inf_{x\in S_{\mu_{Q_{2}}}}$''}(U^{\nu}+Q_{1}) \leq
\text{``$\inf_{x\in S_{\mu_{Q_{2}}}}$''}(U^{\nu}+Q_{2})-C;
$$
hence
$$
U^{\nu}(x)-\text{``$\inf_{x\in S_{\mu_{Q_{2}}}}$''}(U^{\nu}+Q_{2})+C\leq U^{\nu}(x)-
\text{``$\inf_{x\in \Sigma}$''}(U^{\nu}+Q_{1}),\qquad x\in\R^{d}.
$$
Taking the infimum over $\nu$ and making use of (\ref{caract}), one gets the first inequality in (\ref{Robin-monot1}); the second inequality $C\leq F_{2}-F_{1}$ simply follows by letting $x$ go to infinity (here we note that a potential that integrates the Riesz kernel at infinity, tends to $0$ at infinity, up to a set thin at infinity, see Lemma \ref{Mizu}).
The second implication (\ref{Robin-monot2}) can be proved in the same way.
\end{proof}
\end{proposition}

\subsection{Kelvin transform and formulas for the sphere and the ball} \label{KBS}
In the sequel, we will use the Kelvin transform, see e.g. \cite[Chapter IV.5]{Land}. We denote by $T$ the Kelvin transform in $\R^{d+1}$, with respect to the point $y= (y_1,\ldots,y_{d+1})$, $y_{d+1}>0$ and with radius $\sqrt{2 y_{d+1}}$. It maps $\R^d$ onto the sphere $S_{y}^{d}$ of radius 1, having $y$ as its north pole. This transformation was used in \cite{BDT2012} and, recently, in \cite{BDO}. Then, with $x^{*}=T(x)$, $t^{*}=T(t)$, 
 we have 
\begin{equation*}
|x-y| |x^{*}-y| = 2 y_{d+1},\qquad |t^{*}-x^{*}| = \frac{2 y_{d+1} |x-t|}{|x-y|\,|t-y|}\,.
\end{equation*}
For the Riesz potentials and energies, we have the following relations:
\begin{equation}\label{rel-pot}
U^{\mu^{*}}(x^{*})=\frac{|x-y|^{s}}{(2y_{d+1})^{s/2}}U^{\mu}(x),\qquad
I(\mu^{*})=I(\mu),
\end{equation}
where
\begin{equation}\label{rel-meas}
d\mu^{*}(t^{*})=\frac{(2y_{d+1})^{s/2}}{|t-y|^{s}}d\mu(t)=\frac{|t^{*}-y|^{s}}{(2y_{d+1})^{s/2}}d\mu(t).
\end{equation}
Note, in particular, that
\begin{equation}\label{rel-mass-pot}
\mu(\R^{d+1})=(2y_{d+1})^{s/2}U^{\mu^{*}}(y).
\end{equation}

We also recall that the surface area of the $d$-dimensional unit sphere $S^{d}\subset \R^{d+1}$ equals
\begin{equation}\label{surface}
\omega_{d} =\frac{2 {\pi}^{(d+1)/2}}{\Gamma \left({(d+1)}/{2}\right)}.
\end{equation}
The equilibrium measure of the $d$-dimensional unit sphere $S^{d}$ is the normalized surface measure, denoted by $\sigma_{d}$.
Its energy $W(S^{d})$  is given by
\begin{equation}\label{sphereenergy}
W(S^{d}) = \begin{cases} {\displaystyle \frac{\Gamma \left(\frac{d+1}{2}\right)\Gamma (\alpha)}{\Gamma \left(\frac{\alpha+1}{2}\right)\Gamma \left(d-\,\frac{s}{2} \right)}}, & 0<s<d,\,d\geq 3,
\\[10pt]
{\displaystyle 2^{1-s}/(2-s)}, & 0<s<2, \,d=2, \end{cases}
\end{equation}
see \cite[Formula (4.6.5)]{BHS}.
For $d-2 < s < d$ and $d\geq 2$, the equilibrium measure $\omega_R$ of the closed ball $B_R$ of radius $R$ in $\R^{d}$ is absolutely continuous with respect to Lebesgue measure, with density
\begin{equation}\label{equilball}
\omega_R' (x) =  \frac{c_{R}}{(R^2 - |x|^2)^{\alpha/2}},\qquad
c_{R}= \frac{\pi^{-d/2}\Gamma (1+s/2)}{R^{s}\Gamma (1-\alpha/2)},
\end{equation}
see \cite[p.\ 163]{Land}, \cite[Eq.(4.6.12)]{BHS}. Its potential at the point $y=(0,y_{d+1})$ equals
\begin{equation}\label{U-ball}
U^{\omega_{R}}(y)=\int_{\R^{d}}\frac{d\omega_{R}(x)}{|x-y|^{s}}
=
y_{d+1}^{-s}~{}_{2}F_{1}
\left(\frac{s}{2},\frac{d}{2},1+\frac{s}{2},-\frac{R^{2}}{y_{d+1}^{2}}\right),
\end{equation}
which is easily checked from the Euler integral formula for hypergeometric functions. \textcolor{black}{Here $_{2}F_{1}$ 
denotes the Gauss hypergeometric function (see e.g. \cite{Abramowitz})}.
Finally, the $s$-energy of $B_{R}$ is
\begin{equation}\label{ballenergy}
W(B_R) = \frac{s}{2R^{s}}B\left(\frac{s}{2},\frac{\alpha}{2}\right),
\end{equation}
where $B(x,y)$ denotes the Beta function, see \cite[Section 4.6, p.\ 183]{BHS}.
\subsection{Weak balayage and signed equilibrium measures}

The following result plays an important role for our analysis.

\begin{lemma}\label{lem:balayage}
Let $0<s<d$ and
$y= (y_1,\ldots,y_{d+1}) \in \R^{d+1} \setminus \R^d$ with $y_{d+1}\not= 0$.
The weak balayage $\delta_{y}^w$ of $\delta_{y}$ onto $\R^d$ is given by
\begin{equation}\label{pointbalay}
d {\delta}_{y}^w (x) = \frac{(2|y_{d+1}|)^{d-s}}{{\omega_{d}} W(S^{d})|x - y|^{2d-s}}dx,
\end{equation}
where $dx$ denotes the Lebesgue measure in $\R^d$.
Moreover, there is no mass loss in this case; that is,
$\|{\delta}_{y}^w\| = \|\delta_{y}\| = 1.$
Furthermore, if $\nu$ is a signed measure of finite mass with support in $\R^{d+1}\setminus \R^{d}$ (not necessarily compact), its weak balayage $\nu^w$ is given by the superposition
\begin{equation}\label{nubalay}
d \nu^w (x)=\left( \int_{S_\nu} \delta_y^w (x)\, d\nu(y) \right) dx.
\end{equation}
\end{lemma}
\begin{remark} We remark that for $d-2\leq s<d$ the weak balayage measures $\delta_y^w$ and $\nu^w$ coincide with the respective $s$-balayage measures $\widehat{\delta}_{y}$ and $\widehat\nu$ defined in \eqref{balayage} \textcolor{black}{(note that the set of irregular points of ${\mathbb R}^{d+1}\setminus \mathbb{R}^d$ is the empty set)}.
\end{remark}
\begin{proof}
Without loss of generality assume $y_{d+1}>0$. To prove \eqref{pointbalay}, it suffices to verify that
\[U^{\delta_{y}^w}(z) =\int_{\R^d}  \frac{(2y_{d+1})^{d-s}}{{\omega_{d}} W(S^{d})|z-x|^s |x - y|^{2d-s}}dx=\frac{1}{ |z-y|^{s}},\quad z\in\R^{d}.\]
We shall utilize the Kelvin transform $T$ from the previous subsection. The relation between the measure normalized unit surface measure $d\sigma_{d}$ on $S^{d}_{y}$ and the Lebesgue measure on $\R^{d}$ is
$$
\omega_{d}d\sigma_d(x^{*})=\frac{|x^{*}-y|^d}{|x-y|^d}dx =\frac{(2 y_{d+1})^d\, dx}{|x-y|^{2d}},
$$
which yields
\[U^{\delta_{y}^w}(z) =\frac{1}{W(S^{d})|z-y|^s} \int_{S_y^d}  \frac{1}{|z^*-x^*|^s}d\sigma_d(x^*),\quad z^* \in S_y^{d}.\]
Since $\sigma_d(x^*)$ is the equilibrium measure on $S_y^d$, the integral above equals $W(S^d)$ for any $z^* \in S_y^d$, equation \eqref{pointbalay} follows.
The total mass of $\delta_y^w$ is computed as
\[\|\delta_y^w \|=\int_{\R^d}  \frac{(2y_{d+1})^{d-s}}{{\omega_{d}} W(S^{d})|x - y|^{2d-s}}dx=\frac{1}{W(S^d)}\int_{S_y^d}  \frac{1}{|x^*-y|^s}d\sigma_d(x^*) =1.\]
The equation \eqref{nubalay} for any $z\in \R^d$ can be derived as in \cite[Section IV.5, (4.5.5)]{Land}, which completes the proof.
\end{proof}
For our analysis the notion of signed equilibrium measure will be very important.
\begin{definition}
Let $\Sigma$ be a closed subset of $\R^d$. A {\it  signed equilibrium measure} for $\Sigma$ in the external field $Q$ is a \textcolor{black}{(finite)} signed measure $\eta_{Q,\Sigma}$ \textcolor{black}{with finite $s$-energy} supported on $\Sigma$ such that $\eta_{Q,\Sigma} (\Sigma) = 1$ and there exists a finite constant $C_{Q,\Sigma}$ such that
\begin{equation}\label{defsigned}
U^{\eta_{Q,\Sigma}}(x) + Q(x) = C_{Q,\Sigma}\quad\text{q.e. on }\Sigma.
\end{equation}
\end{definition}
\begin{remark}
If this signed equilibrium measure exists, then it is unique, see \cite[Lemma 23]{BDS2009}.
\textcolor{black}{Also, recall that a signed measure $\nu$ has finite energy if and only if both $\nu^+$ and $\nu^-$ have finite energy
(see \cite{Fu}, as well as \cite[Definition 4.2.4, p.134]{BHS} for more general kernels). Thus, it follows from our definition that both $\eta_{Q,\Sigma}^+$ and $\eta_{Q,\Sigma}^-$ have finite energy.}
\end{remark}
Our next result describes the relation between the signed equilibrium and the (positive) equilibrium measure . This result corresponds to \cite[Lemma 3]{KD}, where it was established for the logarithmic kernel in the complex plane, when 
$S_{\mu_Q}$ is a compact set. \textcolor{black}{We consider this result of independent interest and will establish it for a general closed subset $\Sigma$, without imposing the previous assumptions A1) and A2).}
\begin{lemma}\label{lem:inclusion}
Let $d-2\leq s <d$ and let $\Sigma$ be a closed subset of $\R^d$ \textcolor{black}{of positive capacity}.
Assume an equilibrium measure $\mu_{Q}$ and a signed equilibrium measure $\eta_{Q,\Sigma}$ exist. Denote by $\eta_{Q,\Sigma}^+$ the positive part in the Jordan decomposition of $\eta_{Q,\Sigma}$.
 Then,
\\[10pt]
(i) one has
\begin{equation}\label{signeddomin}
\mu_Q \leq \eta_{Q,\Sigma}^+\, ;
\end{equation}
in particular,
\begin{equation*}
S_{\mu_Q} \subseteq S_{\eta_{Q,\Sigma}^+}.
\end{equation*}
(ii) \textcolor{black}{Let $\Sigma_{1}$ be a closed subset of $\Sigma$ that admits a signed equilibrium measure $\eta_{Q,\Sigma_1}$} for which $S_{\mu_Q}\subset\Sigma_{1}$. If $\eta_{Q,\Sigma_{1}}$ is a positive measure, then $\mu_{Q}=\eta_{Q,\Sigma_{1}}$.
\\[10pt]
(iii) Let $\Sigma=\R^{d}$ and let $Q$ be an external field such that $S_{\mu_Q}$ is compact. Assume that there exists an $R_{0}>0$ such that, for each $R$ larger than $R_{0}$, \textcolor{black}{there exists a neighborhood $V$ of the boundary of $B_{R}$, such that the restriction of the signed equilibrium measure $\eta_{Q,B_{R}}$ to $V$ is negative.} Then $S_{\mu_Q}\subset B_{R_{0}}$.
\end{lemma}
\begin{proof}
(i) From \eqref{defsigned} and \eqref{Frostman1}--\eqref{Frostman2} we have that
\begin{align}\label{PAMS}
U^{\eta^+_{Q,\Sigma}}(x) & \leq U^{\eta^-_{Q,\Sigma} + \mu_Q}(x) + C-F_Q,\quad\text{q.e. on } \Sigma,
\\\label{PAMS2}
U^{\eta^+_{Q,\Sigma}}(x)& \geq U^{\eta^-_{Q,\Sigma} + \mu_Q}(x) + C-F_Q,\quad\text{ on }S_{\mu_Q}\textcolor{black}{,}
\end{align}
\textcolor{black}{where \textcolor{black}{$C = C_{Q,\Sigma}$} is the equilibrium constant introduced in (3.24). 
We will prove that $C-F_Q \geq 0$. We consider two cases.} 

\textcolor{black}{First, assume that $\Sigma$ is bounded or thin at infinity. Then, an equilibrium measure $\gamma$ exists normalized so that its potential $U^{\gamma}$ is 1 q.e.\ on $\Sigma$ \textcolor{black}{(see \cite[Sections II.2, V.1]{Land}). Note that $I (\gamma)=\cp(\Sigma)=\gamma(\Sigma)$. 
The measure $\gamma$ is characterized by the following extremal property: its potential $U^\gamma (x)$ is the greatest lower bound of the potentials of measures $\mu$ satisfying $U^\mu (x)\geq 1 $ q.e. on $E$. 
In addition, let us consider for all positive integers $n$ the expanding sequence of compact sets $E_n:=\{ x \in \Sigma \, :\, |x|\leq n\} $ and their equilibrium measures normalized analogously so that $U^{\gamma_n}(x)=1$ q.e. on $E_n$, or $\gamma_n =\cp(E_n)\mu_{E_n}$. Because of the extremal property mentioned above, which characterizes these measures as well, we have that the sequence of potentials $U^{\gamma_n}(x)$ is increasing (see e.g. \cite[Lemma 4.5, p. 275]{Land}) and converges pointwise to the potential of $U^\gamma (x)$ (the argument is similar to \cite[Proof of Theorem 5.1, p. 280]{Land}). Since the $\gamma_n$ have finite energies, \eqref{PAMS} holds $\gamma_n$-a.e., therefore; after integration we get that for all positive integers $n$
\[ \int U^{\eta^+_{Q,\Sigma}}(x) \, d\gamma_n (x) \leq \int U^{\eta^-_{Q,\Sigma} + \mu_Q}(x)\, d\gamma_n(x) + (C-F_Q)\cp(E_n).\]
Applying Fubini-Tonelli's theorem and Lebesgue's monotone convergence theorem (see e.g. \cite{R}), we get
\[\lim_{n\to +\infty} \int U^{\eta^+_{Q,\Sigma}}(x) \, d\gamma_n (x) = \lim_{n\to +\infty} \int U^{\gamma_n}(x) \, d\eta^+_{Q,\Sigma} (x) = \int U^\gamma (x) \, d\eta^+_{Q,\Sigma} (x) =\eta^+_{Q,\Sigma}(\Sigma).\]
In the last equality we used that $\eta_{Q,\Sigma}^+$ has finite $s$-energy. Similarly,
\[\lim_{n\to +\infty} \int U^{\eta^-_{Q,\Sigma}+\mu_Q}(x) \, d\gamma_n (x) = (\eta^-_{Q,\Sigma} +\mu_Q)(\Sigma).\]
Since $\eta^+_{Q,\Sigma}(\Sigma) = (\eta^-_{Q,\Sigma}+\mu_Q)(\Sigma)$ we obtain that
\[0\leq (C-F_Q)\lim_{n\to +\infty}\cp(E_n),\]
which implies that $C-F_Q \geq 0$ (observe that the limit in the last inequality is either finite or $+\infty$).}}

\textcolor{black}{If $\Sigma$ is not thin at infinity, then we can apply \eqref{Mizu1} to $\eta^+_{Q,\Sigma}$ and $(\eta^-_{Q,\Sigma}+\mu_Q)$ with some thin at infinity sets $P_1$ and $P_2$ respectively, such that 
\[ \lim_{|x| \to +\infty, \ x\in \mathbb{R}^d\setminus P_1} U^{\eta^+_{Q,\Sigma}}(x)= 0,\quad  \lim_{|x| \to +\infty,\  x\in \mathbb{R}^d\setminus P_2} U^{\eta^-_{Q,\Sigma}+\mu_Q}(x)=0.\] 
Since \eqref{PAMS} holds q.e., the exceptional set $P_3$,  where the inequality fails is of zero capacity, and hence thin at infinity. Then $P=P_1\cup P_2 \cup P_3$ is thin at infinity. Therefore, there exists a sequence $\{ x_n\} \subset \Sigma\setminus P$, $|x_n|\to +\infty$ as $n\to \infty$, for which \eqref{PAMS} holds and
\[ \lim_{n \to +\infty} U^{\eta^+_{Q,\Sigma}}(x_n)= 0,\quad  \lim_{n \to +\infty} U^{\eta^-_{Q,\Sigma}+\mu_Q}(x_n)=0.\] 
This yields $C-F_{Q}\geq0$ in this case too.} 

\textcolor{black}{This allows us to apply the principle of domination and conclude \eqref{PAMS} holds everywhere (recall that $\eta_{Q,\Sigma}^{+}$ has finite energy), which in turn implies that equality holds in \eqref{PAMS2}. Hence, the assumptions in Theorem \ref{dlvp} \textcolor{black}{with $\mu=\mu_{Q}+\eta_{Q,\Sigma}^{-}$, $\nu=\eta_{Q,\Sigma}^{+}$ and $C=C-F_Q$,}  are satisfied and we thus derive that
$\mu_{Q}+\eta_{Q,\Sigma}^{-}\leq\eta_{Q,\Sigma}^{+}$ on $S_{\mu_Q}$ and consequently $\mu_{Q}\leq\eta_{Q,\Sigma}^{+}$.}\vskip 3 mm
(ii) The fact that $\mu_{Q}=\eta_{Q,\Sigma_{1}}$ is a consequence of (\ref{defsigned}) since, under this hypothesis, the inequalities
(\ref{Frostman4}) and (\ref{Frostman5}) characterizing the equilibrium measure are satisfied by the positive measure $\eta_{Q,\Sigma_{1}}$.
\vskip 3mm
(iii) Assume $S_{\mu_Q}$ is not included in $B_{R_{0}}$. Consider the smallest ball $B_{R}$, $R>R_{0}$, that contains $S_{\mu_Q}$. Since $\mu_{Q}$ satisfies (\ref{Frostman3})--(\ref{Frostman4}) on $B_{R}$, it holds that $\mu_{Q}=\mu_{Q,R}$, the weighted equilibrium measure of $B_{R}$. Now pick some $x\in \partial B_{R}\cap S_{\mu_Q}$. In a small neighborhood $V$ of $x$, we have $\eta_{Q,B_{R}}(V)<0$ while $\mu_{Q,R}(V)>0$ which contradicts (\ref{signeddomin}).
\end{proof}
\section{External fields created by pointwise charges}
\label{examples}
Now, we are concerned with the particular case of external fields as in Corollary \ref{cor:admiss} and Theorem \ref{thm:weakadmiss}, 
created by signed discrete measures $\nu$ supported in $\R^{d+1}\setminus \R^d$,
of the form
\begin{equation}\label{discretemeas}
Q(x):= \,\sum_{j=1}^k\,\gamma_j\,U^{\delta_{y_{j}}}(x)= \sum_{j=1}^k\,\gamma_j\,|x - y_{j}|^{-s}=\sum_{j=1}^k\,\gamma_j\,\left(|x|^2 + y_{j,d+1}^2\right)^{-s/2},
\end{equation}
where $y_{j} := (0; y_{j,d+1})= (0,\ldots,0,y_{j,d+1})$, with $y_{j,d+1}\neq 0$. Without loss of generality, we assume that $y_{j,d+1}>0$, $j=1,\ldots, k$. For the charges, we assume that
$$
\Gamma:= \sum_{j=1}^k \gamma_j\leq -1.$$

In the sequel, the case $\Gamma < -1$, where the compactness of the support of $\mu_Q$ is guaranteed by Corollary \ref{cor:admiss}, will be called admissible (using the analogy with the logarithmic potential setting, see \cite[Chapter I]{ST}), and the case $\Gamma=-1$ will be called weakly admissible (here, only the existence of $\mu_Q$ is ensured a priori by Theorem \ref{thm:general}).

\subsection{Admissible setting}

We will focus on the case of a single ``attractor'', that is $k=1$, and
$$\gamma = \gamma_1 < -1.$$

We suppose, without loss of generality, that $y = (0; y_{d+1})$, with $y_{d+1} >0$.
Then, the external field acting on the hyperplanar conductor $\R^d$ is given by
\begin{equation}\label{efsingle}
Q(x) = \,\frac{\gamma}{(|x|^2 + y_{d+1}^2)^{s/2}}\,,\,x\in \R^d.
\end{equation}
From Theorem \ref{thm:general} 
we know that the equilibrium measure $\mu_Q$ has compact support, extending the result in \cite[Theorem 2.1]{BDO} for the case $d=1$. Since the external field is radial, the support $S_{\mu_Q}$ has circular symmetry, but it may be, in principle, a ball, a sphere or several spheres, or even several shells. Observe that $Q$ is convex in the closed ball
\begin{equation*}
|x|\leq \frac{y_{d+1}}{\sqrt{s+1}},
\end{equation*}
(but not in the whole $\R^d$); hence we can only assert that the intersection of $S_{\mu_Q}$ with that ball is a convex set and, in our case, it is a ball.

We now state our main result in this section.
\begin{theorem}\label{thm:singleattr}
 Let $d-2\leq s<d$. The equilibrium problem in $\R^d$ in the external field \eqref{efsingle} satisfies the following.

\begin{itemize}

\item[(i)] The support $S_{\mu_Q}$ of the equilibrium measure $\mu_Q$ is a closed ball $B_{R_0}$.

\item [(ii)] The density of $\mu_Q$ is given by
\begin{equation}\label{muQ}
\mu'_Q(x) = -\gamma H_{y,R_{0}}(x),\quad H_{y,R_{0}}(x)= \frac{(2y_{d+1})^{\alpha}}{W(S^{d})\omega_{d}}
\left(\frac{1}{|x-y|^{2d-s}}
- \frac{\sin(\alpha\pi/2)}{\pi}J(x,y)
\right),
\end{equation}
where $W(S^{d})$ and $\omega_{d}$ are given in \eqref{sphereenergy} and \eqref{surface}, respectively, and
\begin{equation}\label{integralJ(x)}
J(x,y) 
:=\int_0^{+\infty}\,\frac{u^{\alpha/2-1} du}{(u+1) \left((R_0^2-|x|^2)u+R_0^2+y_{d+1}^2\right)^{d-s/2}}.
\end{equation}
The density of $\mu_{Q}$ vanishes on the boundary of $B_{R_{0}}$.
\item [(iii)] The radius $R_0$ is equal to $R_{0}=y_{d+1}\sqrt{z}$ where $z$ is the unique positive solution of the equation
\begin{equation}\label{sol}
z^{s/2+1}{}_2F_{1}\left(1+\frac{s}{2},1+\frac{d}{2},2+\frac{s}{2},-z\right)
= -\frac{\Gamma(\alpha/2)\Gamma(2+s/2)}{\gamma\Gamma(1+d/2)}.
\end{equation}
In particular, $R_{0}$ is a linear non-decreasing function of the distance $y_{d+1}$.
\end{itemize}
\end{theorem}
\begin{remark}
Theorem \ref{thm:singleattr} extends for general dimension $d$ and $d-2\leq s<d$ the results in \cite[Theorems 2.1 and 2.3]{BDO} for $d=1$ and $0\leq s<1$.
\end{remark}
\begin{figure}[htb]
\centering
  \includegraphics[scale=0.5]{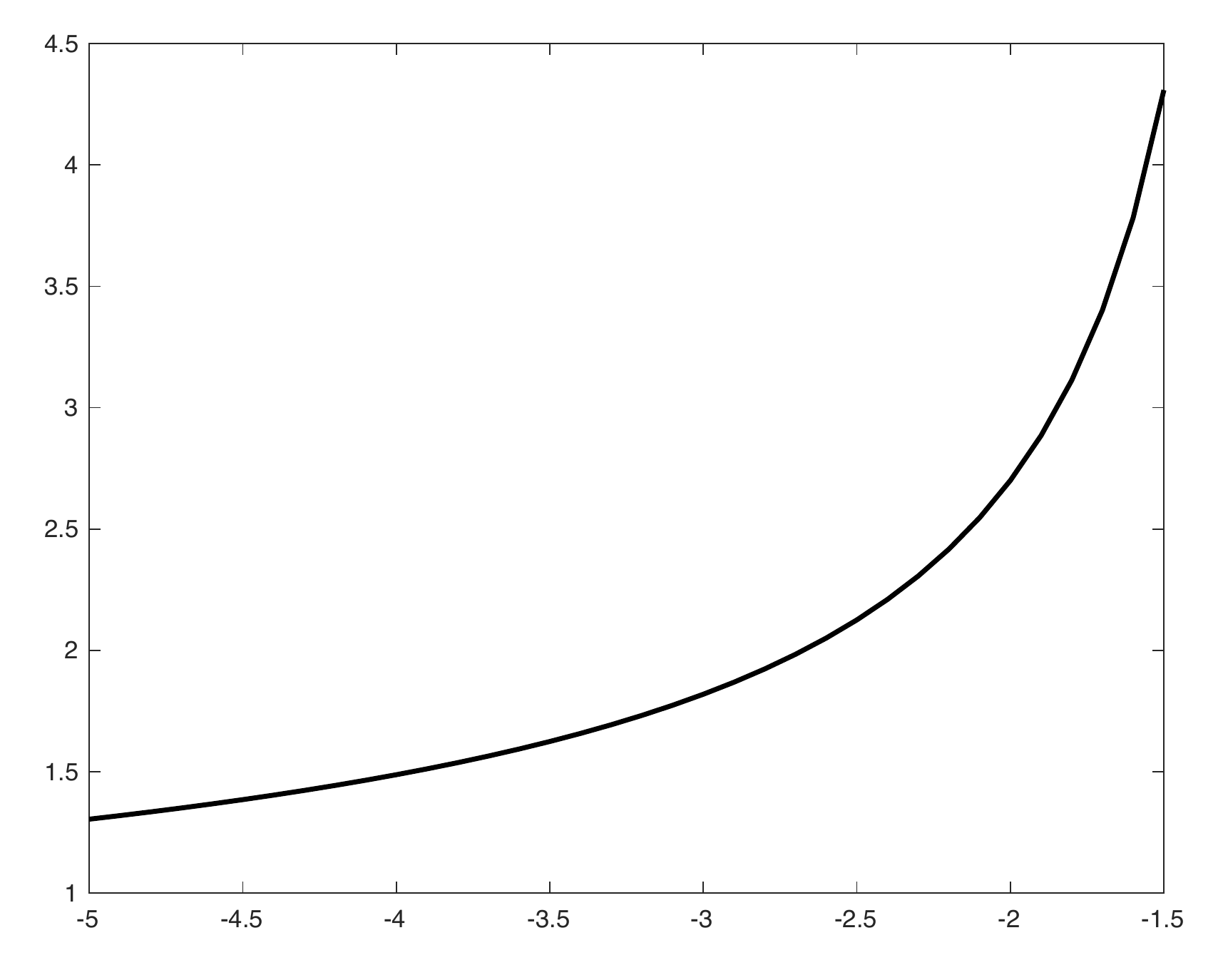}
\caption{The radius $R_0$ of the support $S_{\mu_Q}$ as a function of the attractive charge $\gamma$ ($d=3, s=2, y_{4}=1$)}
\end{figure}
\begin{figure}[htb]
\centering
  \includegraphics[scale=0.5]{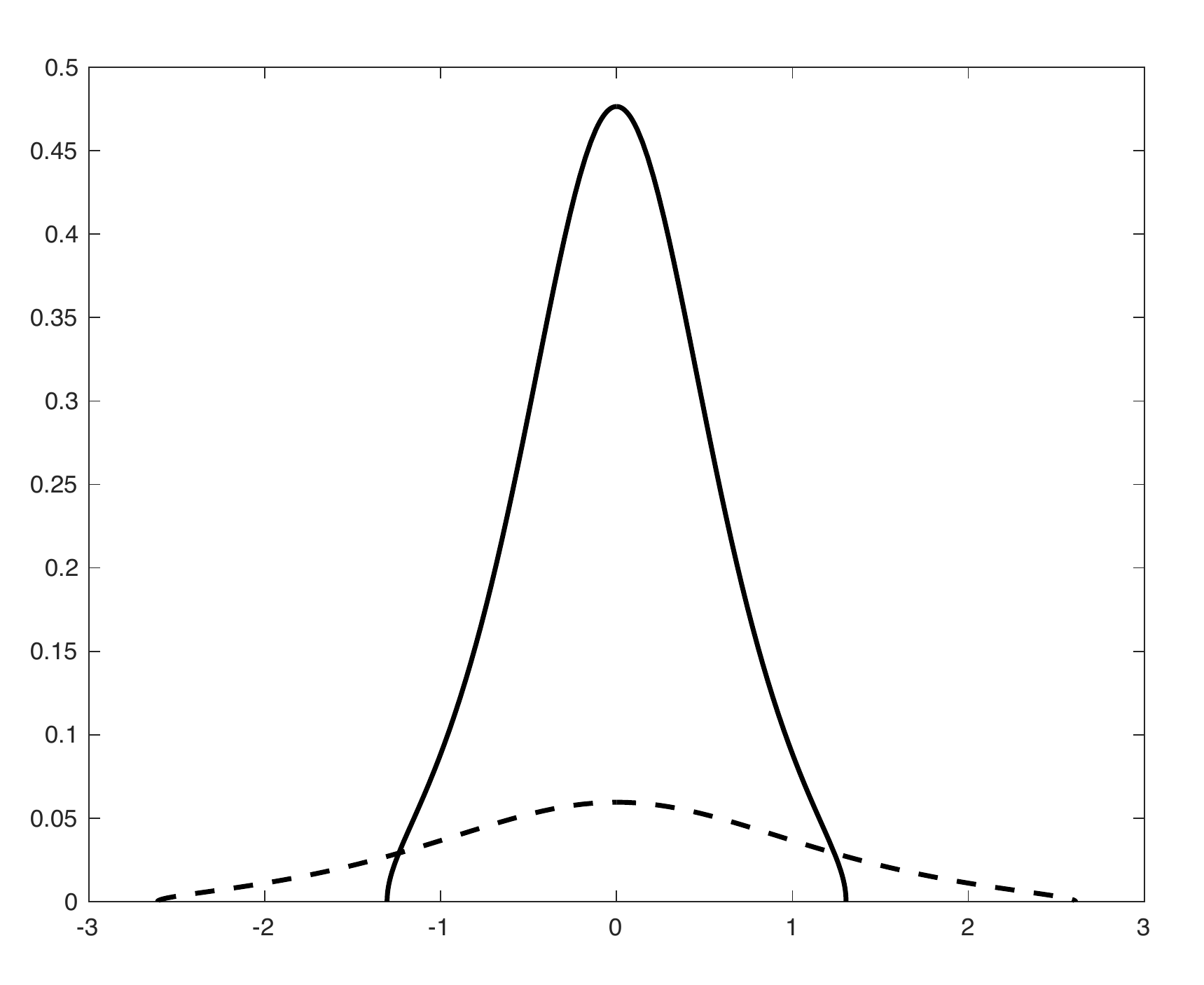}
\caption{Density of the equilibrium measure as a function of $r$ for $d=3$, $s=2$, $\gamma=-5$, $y_{4}=1$ (solid line) and $y_{4}=2$ (dashed line)}
\end{figure}
\begin{remark}\label{rem:Newton}
When $s=d-2$ (Newton case), with $d\geq 3$, \cite[Proposition 2.13]{abey}, which in turn extends \cite[Sec. IV.6]{ST}, yields the explicit expression of $S_{\mu_Q}$ and the density of $\mu_Q$. Observe that in \cite{abey} the growth condition \eqref{Q-infty} is required, but it is just to ensure the existence of $\mu_Q$ and the compactness of $S_{\mu_Q}$; the proof of that result easily follows without assuming \eqref{Q-infty}. Indeed, \cite[Proposition 2.13]{abey} shows that if $d\geq 3$ and $Q(x) = Q(|x|) = Q(r)$, with $r^{d-1}\,Q'(r)$ being a non-decreasing function on $[0,\infty)$, then $S_{\mu_Q} = \{x\in \R^d\,:\,r_0\leq |x|\leq R_0\}$, where $r$ is the smallest value of $r>0$ for which $r^{d-1}\,Q'(r) > 0$ for $r>r_0$, and $R_0$ is the smallest positive solution of the equation
\begin{equation}\label{NewtonR}
R^{d-1}\,Q'(R) = d-2.
\end{equation}
Since for our external field \eqref{efsingle} we have that $r^{d-1}\,Q'(r) = - \gamma s r^d\,(r^2+y_{d+1}^2)^{-1-s/2} > 0\,,$ for each $r>0$, and
\begin{equation*}
(r^{d-1}\,Q'(r))' = - \gamma s d r^{d-1} y_{d+1}^2\,(r^2+y_{d+1}^2)^{-2-s/2} > 0\,,
\end{equation*}
we get $r_0 = 0$ and, consequently, $S_{\mu_Q}$ is a closed ball. To determine the radius $R$ of the ball, \eqref{NewtonR} may be used. Namely, for $d=3, s=1, y_4=1$  and different values of $\gamma$, the following values of $R$ are obtained in Table \ref{Newton-tab}.

\begin{table}[h!]
\begin{center}
\begin{tabular}{|c|c|c|c|c|}
\hline
$\gamma$ & $-10.0$ & $-2.5$ & $-1.1$ & $-1.001$\\[1mm] \hline
$R$ & $ 0.524$ & $1.090$ & $3.90$ & $38.73$ \\[1mm] \hline
\end{tabular}
\caption{The radius $R$ of $S_{\mu_Q}$ for different values of the charge $\gamma$ ($d=3, s=1, y_4=1$)}
\label{Newton-tab}
\end{center}
\end{table}

It is easy to check that, as is natural, the higher the size of the attractive charge, the lowest the radius of the ball. Moreover, the result in \cite{abey} provides the expression for the density of $\mu_Q$, namely, in our case,
$$d\mu_Q(x) = \,\frac{-\gamma s r^{d+1}}{(r^2+y_{d+1}^2)^{2+s/2}}\,dr\,d\sigma_{d-1}(u)\,,\,x=ru\,,\,r=|x|\,,$$
with $d\sigma_{d-1}$ denoting the normalized surface measure of the unit sphere $S^{d-1}$ in $\R^d$.
\end{remark}
\begin{remark}
The results in \cite{abey} were partially extended in the recent paper \cite[Theorem I.1]{Bilog} for more general values of $s$, in particular for $d-2<s<d$; but there the conductor is the unit ball in $\R^d$ and the external field $Q$ is assumed to be convex. Furthermore, in \cite[Theorem I.2]{Bilog} the author obtains the expression of the density of the equilibrium measure, provided that its support $S_{\mu_Q}$ is a ball in $\R^d$. This last result also applies to our Theorem \ref{thm:singleattr}, in particular to the expression for the density of the equilibrium measure in \eqref{muQ}, and one can check the coincidence of \eqref{muQ}--\eqref{integralJ(x)} with \cite[Eq. (8)]{Bilog}. However, there is an important difference in the method of proof which makes this part of Theorem \ref{thm:singleattr} of interest itself: while we need to prove first that $S_{\mu_Q}$ is a ball, and take advantage of this proof to find the density of the equilibrium measure, for the proof of \cite[Theorem I.2]{Bilog} the author obtains that density by handling the Fredholm integral equation of the second kind derived from the Frostman's identity in $S_{\mu_Q}$ (see \eqref{Frostman1}--\eqref{Frostman2}).
Furthermore, though our external field is not convex in the whole $\R^d$ it would be possible to adapt the method of proof of \cite[Theorem I.1]{Bilog} to prove that $S_{\mu_Q}$ is a ball centered at the origin in our case.
\end{remark}

\subsection{Weakly admissible setting}\label{w-Q-charges}

If $\Gamma = -1$, assertion (i) of Theorem \ref{thm:general}  guarantees the existence of the equilibrium measure $\mu_Q$, and condition \eqref{sufficRd} applied to the external field \eqref{discretemeas} provides
\begin{equation*}
\sum_{j=1}^k\,\gamma_j y_{j,d+1}^{\alpha}> 0,
\end{equation*}
as a sufficient condition for the compactness of $S_{\mu_Q}$. Thus, we get infinitely many configurations (indeed, a continuum) of charges and distances $\{(\gamma_j, y_{j,d+1}), j=1,\ldots,k\}$ producing compactly supported equilibrium measures.

\textbf{Case (i). A single attractor at $y=(0;y_{d+1})$, $y_{d+1}>0$ and $\gamma=-1$.}

In this case, $\mu_{Q}=\hat\delta_{y}$, the balayage of $\delta_{y}$ on $\R^{d}$, and from Lemma \ref{lem:balayage}, we know that its support is the entire space $\R^{d}$ (see Remark \ref{rem:weakadmiss} above).

Let us focus now in the simplest non--trivial case of two points, that is, an ``attractor--repellent'' pair.

\textbf{Case (ii). An ``attractor--repellent'' pair }
$$y_{1} = (0; y_{1,d+1}), \quad y_{2} = (0; y_{2,d+1})\quad\text{ with }\quad\gamma_1  = -1 -\gamma, \quad\gamma_2=\gamma>0.
$$
Now, the external field is given by
\begin{equation}\label{extfieldattrep}
Q(x) = -\frac{1+\gamma}{(|x|^2+y_{1,d+1}^2)^{s/2}}+\frac{\gamma}{(|x|^2+y_{2,d+1}^2)^{s/2}},
\end{equation}
and, by Lemma \ref{lem:balayage}, the density of the signed equilibrium measure on $\R^{d}$ is given by
\begin{equation}\label{signedattrep}
\eta_{Q,\R^{d}}'(x)=\frac{2^{\alpha}}{{\omega_{d}} W(S^{d})}
\left(\frac{(1+\gamma)y_{1,d+1}^{\alpha}}{|x - y_{1}|^{2d-s}}-\frac{\gamma y_{2,d+1}^{\alpha}}{|x - y_{2}|^{2d-s}}\right).
\end{equation}
We set
\begin{equation}\label{def-g-R}
g=\frac{\gamma}{1+\gamma},\qquad
\rho=\left[g\left(\frac{y_{2,d+1}}{y_{1,d+1}}\right)^{\alpha}\right]^{2/(2d-s)},\qquad
R=\left(\frac{\rho y_{1,d+1}^{2}-y_{2,d+1}^{2}}{1-\rho}\right)^{1/2}.
\end{equation}
where, in the definition of $R$, it is assumed that the argument of the square root is nonnegative.
\begin{theorem}\label{thm:weakly}
 Let $d-2\leq s<d$. Regarding the equilibrium problem in $\R^d$ in the weakly admissible external field (4.7) three different cases arise according to the value of the quotient $y_{2,d+1}/y_{1,d+1}$:
\begin{itemize}
\item[(i)] If $y_{2,d+1}/y_{1,d+1}\in(0,g^{1/d})$, $S_{\mu_Q} \subset B_R^{c}$,  where $B^{c}_{R}$ denotes the complement of the ball $B_{R}$, and the radius $R$ is given by \eqref{def-g-R}.
\item[(ii)] If $y_{2,d+1}/y_{1,d+1}\in[g^{1/d},g^{-1/\alpha}]$,
then $S_{\mu_Q} = \R^d$ and $\mu_Q = \eta_{Q,\R^d}$.

\item[(iii)] If $y_{2,d+1}/y_{1,d+1}\in(g^{-1/\alpha},\infty)$, $S_{\mu_Q} \subset B_R$, with $R$ also given by \eqref{def-g-R}.
\end{itemize}
\end{theorem}
\begin{remark}
We conjecture that in case (i), $S_{\mu_Q}$ is the complement of a ball of radius larger than $R$, and that in case (iii), $S_{\mu_Q}$ is a ball of radius smaller than $R$.
Note that the case where $y_{2,d+1}/y_{1,d+1} = 1$, that is, $y_{2,d+1}=y_{1,d+1}$, is included in part (ii) of the previous theorem. In this case, part of the negative charge cancels with the positive one and it results in a single attractor of charge $-1$ as in case (i) at the beginning of this section. Furthermore, the conclusion is the same if the attractor and the repellent are placed at the same distance of the hyperplanar conductor $\R^d$ but on different half-hyperplanes of $\R^{d+1}$.
\end{remark}
For an illustrative purpose,
in Figure \ref{fig:signed} the density of the positive part of the signed equilibrium measure of $\R^2$, with $s=1$, in the presence of the external field created by an attractor of charge $-2$ placed at $(0,0,1)\in \R^3$ and a repellent of charge $1$ located at $(0,0,3)$ is shown; also, in Figure \ref{fig:signedradius} it is plotted that density as a function of $r = \sqrt{x^2+y^2}$. It is clear that the support of the positive part of $\eta_Q$ is a compact subset of $\R^2$. Thus, Lemma \ref{lem:inclusion} implies that $S_{\mu_Q}$ is also a compact subset of $\R^2$. In this case, part (iii) in Theorem \ref{thm:weakly} applies and we have that $S_{\mu_Q} \subset B_R$, with $R = 4.978$.
\begin{figure}[h!]
    \begin{center}
    \includegraphics[scale=.6]{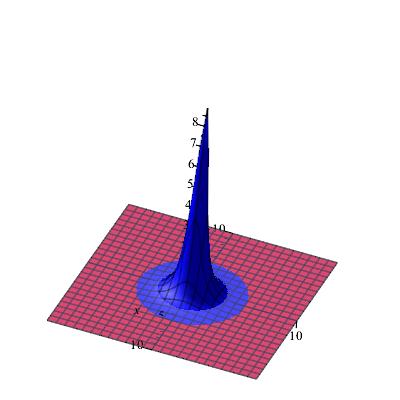}
    \end{center}
    \caption{Density of the positive part of the signed equilibrium measure (blue). In red, the plane $\R^2$.}
    \label{fig:signed}
\end{figure}

\begin{figure}[h!]
    \centering
    \includegraphics[scale=.5]{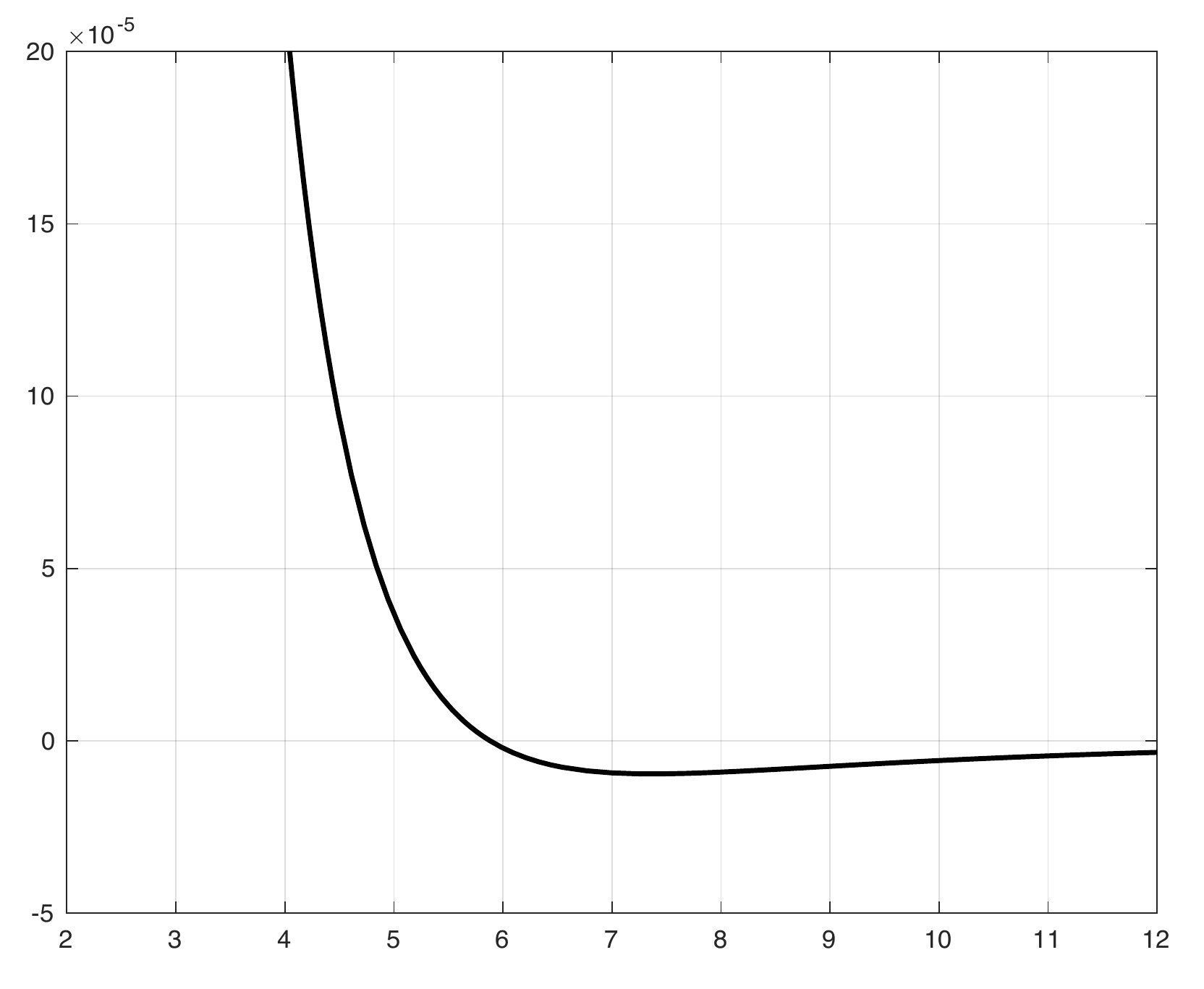}
    \caption{Density of the signed equilibrium measure $\eta_{Q,\R^{d}}$ as a function of $r$, near the point where it vanishes.}
    \label{fig:signedradius}
\end{figure}

Numerically one can check the validity of the conjectures made after Theorem \ref{thm:weakly}.
In that connection, Figure \ref{Rvscharge_y2} shows the radius $R_{0}$ of the ball $S_{\mu_Q}$ (case (iii) of Theorem \ref{thm:weakly}) as a function of the charge $\gamma$ and as a function of $y_{2,4}$ with $d=3$, $s=2$ and $y_{1,4}$=1.
\begin{figure}[htb]
\centering
\includegraphics[scale=0.4]{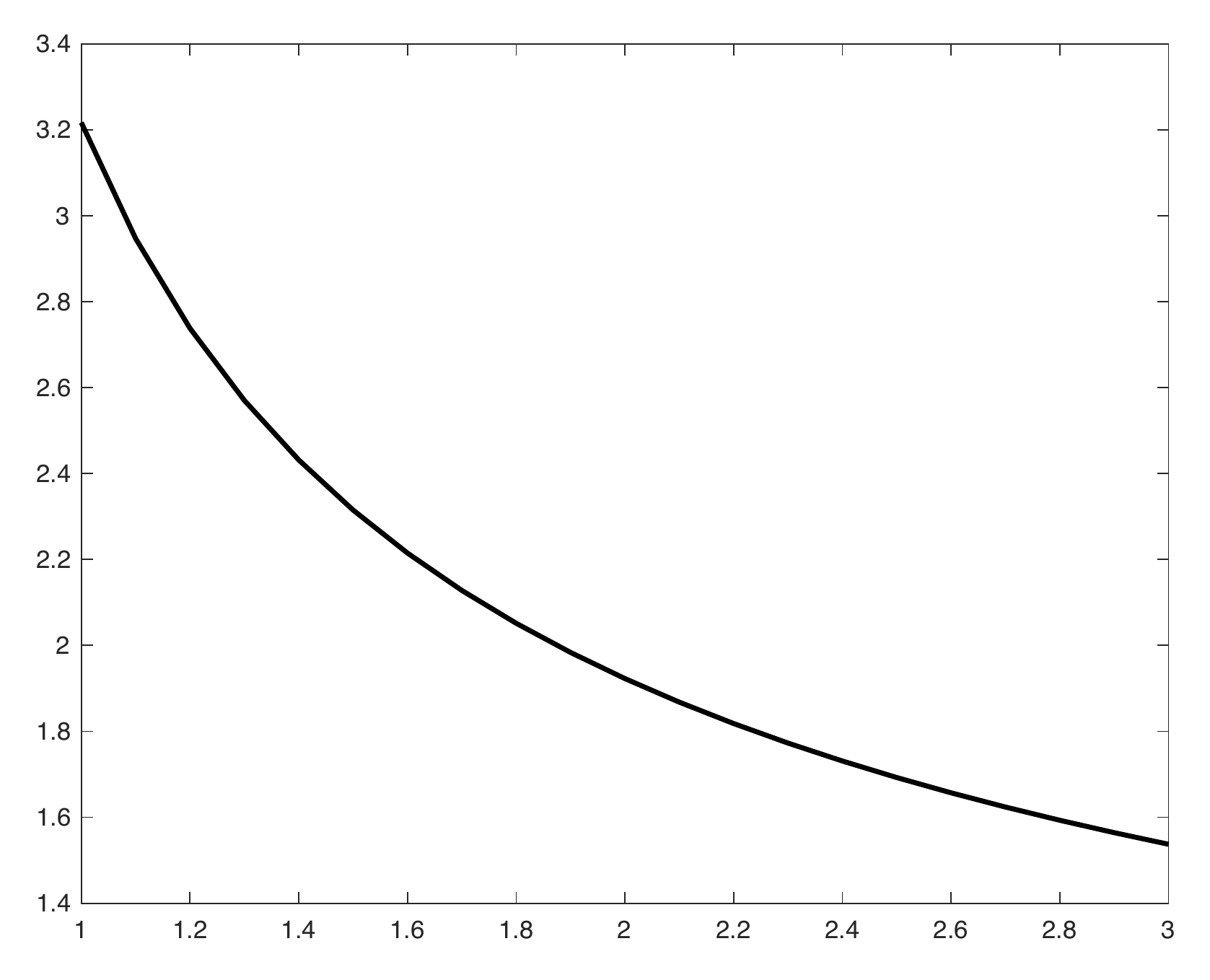}
\includegraphics[scale=0.4]{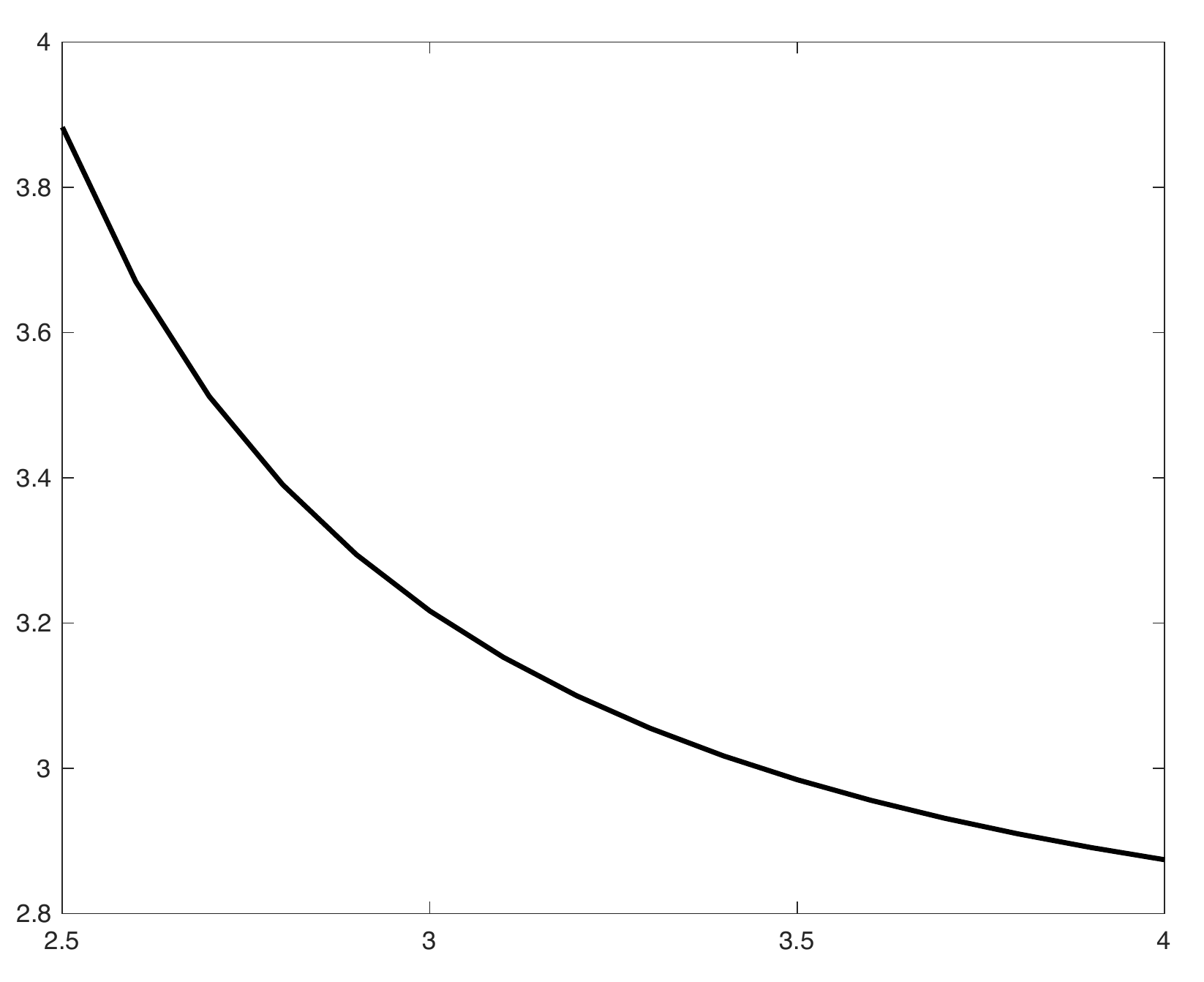}
\caption{The radius $R_0$ of the support $S_{\mu_Q}$ as a function of $\gamma$ (left) and as a function of $y_{2,4}$ (right) ($d=3, s=2, y_{1,4}=1$)}
\label{Rvscharge_y2}
\end{figure}
\begin{figure}[htb]
\centering
  \includegraphics[scale=0.5]{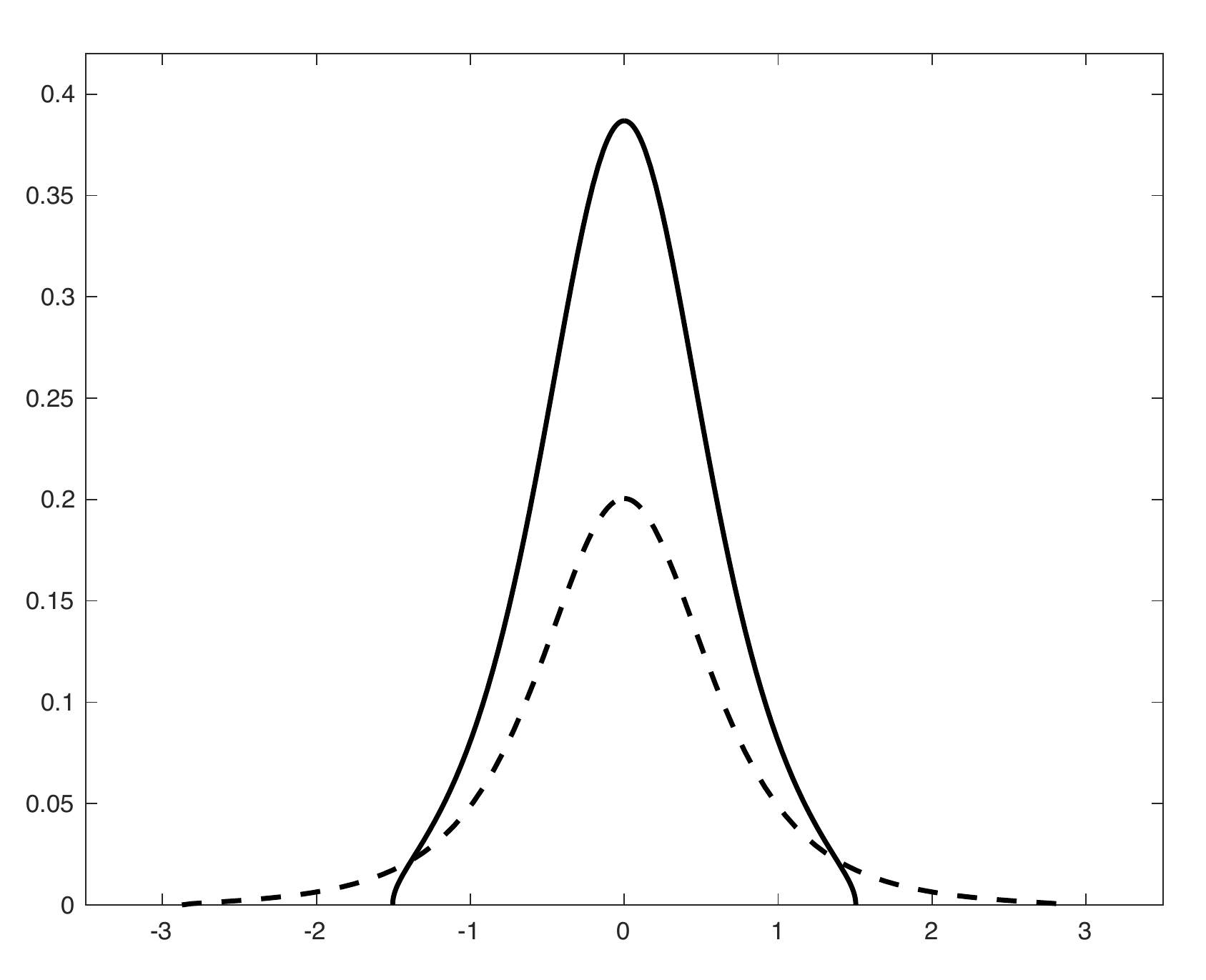}
\caption{Density of the equilibrium measure as a function of $r$ for $d=3$, $s=2$, $y_{1,4}=1$, $y_{2,4}=4$, $\gamma=3$ (solid line) and $\gamma=1$ (dashed line)}
\end{figure}
\section{Proofs}
\label{proofs}
\subsection{Proofs of Theorems \ref{equiv-Min-Fro}, \ref{thm:general} and Corollary \ref{cor:deltaorigin}}
We start with the proof of Theorem \ref{equiv-Min-Fro}.
\\[\baselineskip]
{\bf Proof of Theorem \ref{equiv-Min-Fro}.}
{\bf (i)} The inequality $-\infty<W_{Q}(\Sigma)$ holds because the unweighted energy is positive and $Q$ is lower-bounded on $\Sigma$. Moreover, since $\{x\in\Sigma,~Q(x)<\infty\}$ has positive capacity, there exists an $n\in\N$ such that $\Sigma_{n}:=\{x\in\Sigma,~Q(x)<n\}$ has the same property. Hence, there exists $\mu_{n}\in\PP(\Sigma_{n})$ such that $I_{Q}(\mu_{n})<\infty$.
\\
{\bf (ii)} The uniqueness of a minimizing measure can be proved as usual, see e.g. \cite[Theorem I.1.3]{ST}, based on the fact that for two measures $\mu$ and $\nu$, $I(\mu-\nu)=0$ if and only if $\mu=\nu$, see \cite[Theorem 1.15]{Land}.
\\
{\bf (iii)} The Frostman inequalities can be proved by following the arguments of \cite[Theorem I.1.3]{ST} for weighted logarithmic potentials.
\\
{\bf (iv)} For the characterization of the equilibrium measure, we follow the proof of \cite[Proposition 2.6]{BLW2} : assume $\mu$ satisfies (\ref{Frostman3})-(\ref{Frostman4}).
Pick any $\nu\in\PP(\Sigma)$, with $I_{Q}(\nu)<\infty$.
Writing
$\nu=\mu+(\nu-\mu)$, we have
\begin{equation}\label{decomp}
I_{Q}(\nu)=I_{Q}(\mu)+I(\nu-\mu)+2\int_{\Sigma}(U^{\mu}+Q)(d\nu-d\mu),
\end{equation}
where the right-hand side is well-defined since $Q$ is lower-bounded on $\Sigma$, the energies and weighted energies of $\mu$ and $\nu$ are finite, and the mixed energy satisfies $I(\mu,\nu)\leq I(\mu)I(\nu)$, see \cite[p.\ 82]{Land}. Making use of the Frostman inequalities for $\mu$, we obtain
\begin{equation}\label{ineq-nu-mu}
\int(U^{\mu}+Q)(d\nu-d\mu)\geq F\int d\nu-F\int d\mu=F(\nu-\mu)(\R^{d})=0.
\end{equation}
Moreover, by \cite[Theorem 1.15]{Land}, $I(\nu-\mu)\geq0$. Hence, from (\ref{decomp}) we derive that $I_{Q}(\nu)\geq I_{Q}(\mu)$. We conclude that $\mu$ is a minimizing measure, and, finally, $F=F_{Q}$ by uniqueness. For the last assertion, one may restrict $\nu$ to be in $\PP(S_{\mu_Q})$. Then, (\ref{ineq-nu-mu}) still holds and the end of the argument remains the same.
\qed
\\[\baselineskip]
We now proceed with preparations for the proof of Theorem \ref{thm:general}. We will make use of the Kelvin transform $T$ as defined in Section \ref{KBS}. Here, we choose $T$ of
radius $\sqrt{2}$ and center $y=(0,0,\ldots,1)$ in $\R^{d+1}$, which sends $\R^{d}$ onto the $d$-dimensional sphere $S$ in $\R^{d+1}$, centered at 0 of radius 1.
We keep the notation $x^{*}=T(x)$ so that, in particular,
\begin{equation}\label{corresp-x-x*}
|x-y||x^{*}-y|=2.
\end{equation}
Note that \textcolor{black}{$\Sigma^{*}=T(\Sigma)\cup\{y\}$} is a closed subset of $S$, hence a compact set in $\R^{d+1}$.

We first check how the minimization problem \eqref{weightenergy} and the Frostman inequalities \eqref{Frostman1}-\eqref{Frostman2} translate when we apply the Kelvin transform. We know from
(\ref{rel-pot})-(\ref{rel-meas}) that
for Riesz potentials and energies, we have the following relations,
$$
U^{\mu^{*}}(x^{*})=2^{-s/2}|x-y|^{s}U^{\mu}(x),\quad I_{s}(\mu^{*})=I_{s}(\mu),$$
where
$$d\mu^{*}(t^{*})=2^{s/2}\frac{d\mu(t)}{|t-y|^{s}}=2^{-s/2}|t^{*}-y|^{s}d\mu(t).
$$
Note that, by (\ref{rel-mass-pot}), the condition $\mu\in\PP(\Sigma)$ translates into
$2^{s/2}U^{\mu^{*}}(y)=1$.

Then, with $Q^{*}(x)=Q(x^{*})$, the minimization problem (\ref{weightenergy}) becomes
\begin{equation*}
\min\left(I_{s}(\mu^{*})+2\int \frac{2^{s/2}Q^{*}(t)}{|t-y|^{s}}d\mu^{*}(t)\right),
\end{equation*}
where the minimum is taken over all measures $\mu^{*}$ supported on $\Sigma^{*}$ such that $2^{s/2}U^{\mu^{*}}(y)=1$.
The Frostman inequalities (\ref{Frostman1})-(\ref{Frostman2}) on $\Sigma$ become the following ones on $\Sigma^{*}$,
\begin{align}
U^{\mu_{Q}^{*}}(x)+2^{s/2}\frac{Q^{*}(x)-F}{|x-y|^{s}} & \geq 0,\quad\text{q.e. on }\Sigma^{*}, \label{F3}\\[5pt]
U^{\mu_{Q}^{*}}(x)+2^{s/2}\frac{Q^{*}(x)-F}{|x-y|^{s}} & \leq 0, \quad x\in S_{\mu_Q}^{*}=S_{\mu_{Q}^{*}}, \label{F4}
\end{align}
\textcolor{black}{where now we denote $F = F_Q$, the Robin constant}. As a preliminary result, we show the existence, under some assumptions on $Q^{*}$, of a measure $\mu^{*}$ on $\Sigma^{*}$ satisfying $2^{s/2}U^{\mu^{*}}(y)=1$ and the above Frostman inequalities.
\begin{lemma}\label{prelim}
Assume
\begin{equation}\label{cond-Q*}
Q^{*}(x)\leq Q^{*}(y)-2^{-s}|x-y|^{s},\quad\text{in a neighborhood of }y,
\end{equation}
or
\begin{equation}\label{cond2-Q*}
\lim_{x\to y,~x\in\Sigma^{*}\setminus P} 2^{s}\frac{Q^{*}(x)- Q^{*}(y)}{|x-y|^{s}}\leq-1,
\end{equation}
where $P$ is thin at $y$.
Then there exists a constant $F\leq Q^{*}(y)$ and a measure $\mu^{*}$ on the compact set $\Sigma^{*}$ with $2^{s/2}U^{\mu^{*}}(y)=1$ such that
\begin{align}
U^{\mu^{*}}(x)+2^{s/2}\frac{Q^{*}(x)-F}{|x-y|^{s}} & \geq 0,\quad\text{q.e. on }\Sigma^{*}, \label{F3bis}
\\[5pt]
U^{\mu^{*}}(x)+2^{s/2}\frac{Q^{*}(x)-F}{|x-y|^{s}} & \leq 0, \quad x\in S_{\mu^{*}}. \label{F4bis}
\end{align}
\end{lemma}
\begin{remark}
Recall that (\ref{F3bis})-(\ref{F4bis}) on $\Sigma^{*}$ correspond, via the Kelvin transform, to \eqref{Frostman1}-\eqref{Frostman2} on $\Sigma$.
\end{remark}
\begin{proof}[Proof of Lemma \ref{prelim}]
We will obtain the measure $\mu^{*}$ by considering the weak-* limit of a sequence of measures $\mu_{n}^{*}$ solving the problem on the set $\Sigma_{n}^{*}$, obtained from $\Sigma^{*}$ by removing its intersection with the open ball centered at $y$ with radius $1/n$.
The fact that there exists a measure $\mu_{n}^{*}$ on $\Sigma_{n}^{*}$ with $2^{s/2}U^{\mu_{n}^{*}}(y)=1$  satisfying for some $F_{n}$ the Frostman inequalities
\begin{align}
U^{\mu^{*}_{n}}(x)+2^{s/2}\frac{Q^{*}(x)-F_{n}}{|x-y|^{s}} & \geq 0,\quad\text{q.e. on }\Sigma^{*}_{n}, \label{Fn1}
\\[5pt]
U^{\mu^{*}_{n}}(x)+2^{s/2}\frac{Q^{*}(x)-F_{n}}{|x-y|^{s}} & \leq 0, \quad x\in S_{\mu^{*}_{n}}, \label{Fn2}
\end{align}
just follows by considering the corresponding problem, via the Kelvin transform, on the compact set $\Sigma_{n}=T(\Sigma_{n}^{*})$ (this just uses the fact that $Q(x)=Q^{*}(x^{*})$ is \textcolor{black}{lower semicontinuous} on $\Sigma_{n}$).
Next, since
$$
\frac{\|\mu_{n}^{*}\|}{\diam(\Sigma^{*})^{s}}\leq U^{\mu_{n}^{*}}(y)=2^{-s/2},
$$
the masses of the $\mu_{n}^{*}$ are uniformly bounded, and we may consider a subsequence (still denoted by $\mu_{n}^{*}$) which converges weak-* to some measure $\mu^{*}$. The $F_{n}$'s are lower bounded since, from (\ref{Fn2}) and the fact that a Riesz potential is positive, we have, for $x\in S_{\mu^{*}_{n}}$,
$$
-\infty<\inf_{x\in \Sigma^{*}}Q^{*}(x)\leq Q^{*}(x)\leq F_{n}\,.
$$
They are also upper bounded. Indeed, we know that $U^{\mu^{*}}$ is finite q.e. and, by the \textcolor{black}{ lower envelope theorem \textcolor{black}{(see \cite[Theorem 11]{Br2} and \cite[Theorem 3.8, p. 190]{Land}})}
$$
\liminf_{n}U^{\mu_{n}^{*}}(x)=U^{\mu^{*}}(x)\quad\text{q.e.} \ {\rm \textcolor{black}{on}}\ \R^{d}.
$$
Thus, in view of (\ref{Fn1}), and possibly by considering a subsequence, we can find some $x_{0}\neq y\in \Sigma^{*}$ such that \textcolor{black}{for all $n$}
$$
\frac{F_{n}}{|x_{0}-y|^{s}}\leq 2^{-s/2}U^{\mu^{*}}(x_{0})+\frac{Q^{*}(x_{0})}{|x_{0}-y|^{s}} +1<\infty,
$$
(here we also use that $\{x\in \Sigma^{*},~Q^{*}(x)<\infty\}$ is of positive capacity).
Consequently, we may again consider a subsequence so that $F_{n}$ tends to some constant $F$ as $n$ goes large. Taking the limit in (\ref{Fn1}) and making use of the lower envelope theorem \textcolor{black}{ and the countable subadditivity of the capacity on Borel sets},  we get \textcolor{black}{\eqref{F3bis}.}

\textcolor{black}{To prove} \eqref{F4bis}, recall from \cite[Eq.(0.1.10)]{Land} that
$$
S_{\mu^{*}}\subset\bigcap_{N=1}^{+\infty}\bar{\bigcup_{n=N}^{+\infty}S_{\mu_{n}^{*}}}.
$$
Hence, for $x\in S_{\mu^{*}}$, we can find a sequence $x_{n}\to x$ with $x_{n}\in S_{\mu_{n}^{*}}$. Making use of the principle of descent, we get
$$
U^{\mu^{*}}(x)+2^{s/2}\frac{Q^{*}(x)-F}{|x-y|^{s}}\leq \liminf_{n}U^{\mu^{*}_{n}}(x_{n})+\liminf_{n}2^{s/2}\frac{Q^{*}(x_{n})-F_{n}}{|x_{n}-y|^{s}}
\leq 0.
$$
Next, we show that $F\leq Q^{*}(y)$. For some set $A_{0}$ thin at $y$, we have
$$
\lim_{x\to y,~x\not\in A_{0}}U^{\mu^{*}}(x)=U^{\mu^{*}}(y)\leq\liminf_{n}U^{\mu_{n}^{*}}(y)=2^{-s/2},
$$
where we refer to \cite[Theorem 5.1, p.\ 79]{M} for the equality and to the principle of descent for the inequality.
Now, rewriting (\ref{F3bis}) as
$$
U^{\mu^{*}}(x)+2^{s/2}\frac{Q^{*}(x)-Q^{*}(y)}{|x-y|^{s}}+2^{s/2}\frac{Q^{*}(y)-F}{|x-y|^{s}} \geq 0,\quad\text{q.e. on }\Sigma^{*},
$$
we get, together with (\ref{cond-Q*}) or (\ref{cond2-Q*}), that the sum of the first two terms is less than or equal to $0$ q.e. near $y$. The inequality $Q^{*}(y)\geq F$ follows.

Finally we show that $2^{s/2}U^{\mu^{*}}(y)=1$. Assume to the contrary that $2^{s/2}U^{\mu^{*}}(y)<1$. Then, the previous reasoning shows that $Q^{*}(y)>F$.
\textcolor{black}{But (\ref{Fn2}) entails that
$$
Q^{*}(x)-F_{n}\leq|x-y|^{s},\qquad x\in S_{\mu_{n}^{*}}.
$$
From the continuity of $Q^{*}$ at $y$ and the convergence of $F_{n}$ to $F$, we derive that, for $x$ close to $y$, and large $n$, 
$$
0<\frac12(Q^{*}(y)-F)\leq|x-y|^{s},\qquad x\in S_{\mu_{n}^{*}},
$$
showing, that for $n$ large, $S_{\mu_{n}^{*}}$ does not contain a fixed ball centered at $y$.
Hence 
$$
2^{s/2}U^{\mu^{*}}(y)=\lim_{n}2^{s/2}U^{\mu_{n}^{*}}(y)=1,
$$ 
a contradiction.}
\end{proof}
\textcolor{black}{We will also make use of the following elementary lemma.
\begin{lemma}\label{elem-lem}
Let $\mu$ be a probability measure and $U^{\mu}$ be its potential. Let $\eps>0$. Then there exists $C>0$ such that $|x|^{s}U^{\mu}(x)\geq1-\eps$ when $|x|\geq C$.
\end{lemma}
\begin{proof}
Making use of the Kelvin transform introduced before (\ref{corresp-x-x*}), we have, see (3.14) and (3.15),
$$
|x|^{s}U^{\mu}(x)=\frac{|x|^{s}}{|x-y|^{s}}2^{s/2}U^{\mu^{*}}(x^{*})$$
and
$$
2^{s/2}U^{\mu^{*}}(y)=2^{s/2}\int\frac{d\mu^{*}(t^{*})}{|t^{*}-x|^{s}}
=\int d\mu(t)=1.
$$
By lower semicontinuity of the potential $U^{\mu^{*}}$, for $\eps>0$, there is a neighborhood $V_{y}$ of $y$ such that
$$
\forall x^{*}\in V_{y},\quad 2^{s/2}U^{\mu^{*}}(x^{*})\geq 2^{s/2}U^{\mu^{*}}(y)-\eps/2=1-\eps/2.
$$
Applying the Kelvin transform again, we get a neighborhood of the point at infinity, namely some $C\geq0$ such that for $|x|\geq C$,
$$
|x-y|^{s}U^{\mu}(x)=(|x|^{2}+1)^{s/2}U^{\mu}(x)\geq1-\eps/2.
$$
Hence,
$$
(1+1/C^{2})^{s/2}|x|^{s}U^{\mu}(x)\geq1-\eps/2
$$
and the assertion of the lemma follows by choosing $C$ large enough.
\end{proof}
} 
\noindent
{\bf Proof of Theorem \ref{thm:general}.}
{\bf (i)} The Riesz energy problem on $\Sigma$ and the associated Frostman inequalities are equivalent, via the Kelvin transform, to the Frostman inequalities on $\Sigma^{*}$, considered in Lemma \ref{prelim}.
Moreover, from (\ref{corresp-x-x*}), we see that assumptions (\ref{cond-Q*}) or  (\ref{cond2-Q*})
are respectively equivalent, under the Kelvin transform, to
\begin{equation*}
Q(x)\leq Q(\infty)-|x-y|^{-s} = Q(\infty)-(|x|^{2}+1)^{-s/2},\quad\text{in a neighborhood of }\infty,
\end{equation*}
or
\begin{equation*}
\lim_{|x|\to+\infty,~x\in\Sigma\setminus P} |x-y|^{s}(Q(x)- Q(\infty))
=\lim_{|x|\to+\infty,~x\in\Sigma\setminus P} (|x|^{2}+1)^{s/2}(Q(x)- Q(\infty))\leq-1,
\end{equation*}
where $P$ is thin at $\infty$.
Since these conditions are slightly weaker than (\ref{cond-Q}) and (\ref{cond2-Q}), i) follows from Lemma \ref{prelim} and assertion iv) in Theorem \ref{equiv-Min-Fro}.
\\[.5\baselineskip]
{\bf (ii)} Assume that $\mu_{Q}$ has unbounded support (recall it satisfies \eqref{Frostman1}-\eqref{Frostman2}). From Remark \ref{Rem-Frost} we infer that $Q(\infty)\geq F_{Q}$.
Also, \eqref{Frostman2} implies that $Q(x)\leq F_{Q}$ for $x\in S_{\mu_{Q}}$, and since $S_{\mu_{Q}}$ is unbounded, we derive that $Q(\infty)\leq F_{Q}$. Thus, $Q(\infty)=F_{Q}$.

\textcolor{black}{Next, multiplying inequality \eqref{Frostman1} by $|x|^{s}$ and making use of the facts that,
outside of a set thin at infinity,
$\lim_{|x|\to+\infty}|x|^{s}U^{\mu_{Q}}(x)=1$, see Lemma \ref{Mizu}, 
and that a polar set is thin at infinity,
we get 
$$
K_{Q}:=\lim_{|x|\to+\infty,~x\in\Sigma}|x|^s\left(Q(x) - Q(\infty)\right) \geq -1.
$$
Moreover, from \eqref{Frostman2}
follows that
\begin{align*}
0 & \geq\liminf_{|x|\to+\infty,~x\in S_{\mu_{Q}}}|x|^{s}U^{\mu_{Q}}(x)+K_Q
= \liminf_{|x|\to+\infty,~x\in S_{\mu_{Q}}}\int\frac{|x|^{s}}{|t-x|^{s}}d\mu_{Q}(t)+K_Q
\\[5pt]
& \geq \int\liminf_{|x|\to+\infty,~x\in S_{\mu_{Q}}}\frac{|x|^{s}}{|t-x|^{s}}d\mu_{Q}(t)+K_Q
=1+K_Q,
\end{align*}
where in the second inequality, we have used Fatou's lemma (see e.g.\ \cite{R}), and in the last equality, the fact that $S_{\mu_{Q}}$ is unbounded. Item (ii) is thus proved.
}
\\[.5\baselineskip]
{\bf (iii)} Assume now that $d-2\leq s<d$ and that a probability measure $\mu_{Q}$ exists that satisfies the Frostman inequalities on $\Sigma$, where we substract $Q(\infty)$ on both sides:
\begin{align*}
U^{\mu_{Q}}(x)+Q(x)-Q(\infty) & \geq F_{Q}-Q(\infty),\quad\text{q.e. on }\Sigma, \\[5pt]
U^{\mu_{Q}}(x)+Q(x)-Q(\infty) & \leq F_{Q}-Q(\infty), \quad x\in S_{\mu_Q}.
\end{align*}
 \textcolor{black}{ Using Remark \ref{Rem-Frost} again, we have that  $Q(\infty)\geq F_{Q}$.} Then, from the second inequality, the assumption (\ref{cond-Q2}) evaluated near infinity \textcolor{black}{and Lemma \ref{elem-lem}},
we derive that
$S_{\mu_Q}$ must be bounded. Thus, comparing external fields, we have, for some $\epsilon>0$,
$$-\frac{1}{|x|^{s}}+\epsilon<-\frac{c}{|x|^{s}}\leq Q(x)-Q(\infty),\quad x\in S_{\mu_Q}.$$
In the case of the external field $-1/|x|^{s}$ we know from assertion i) that a minimizing measure exists, and Corollary \ref{cor:deltaorigin} shows (see below) that the corresponding Robin constant is zero.
Thus, together with the property of strict monotonicity of Proposition \ref{monot}, we obtain that
$F_{Q}-Q(\infty)>0$, a contradiction. Hence, no probability measure on $\Sigma$ satisfies Frostman inequalities and, by the direct implication in Theorem \ref{equiv-Min-Fro}, no measure minimizing the weighted energy $I_{Q}$ exists on $\Sigma$.
\qed
\\[\baselineskip]
{\bf Proof of Corollary \ref{cor:deltaorigin}.}
Corollary \ref{cor:deltaorigin} is a consequence of Theorem \ref{thm:general}, except for the fact that the proof of assertion iii) of the theorem relies on the property that $F_{Q}=0$ when $Q(x)=-1/|x|^{s}$. So, in that case, we
give two direct proofs, of possible independent interest, of the existence of a minimizing measure $\mu_{Q}$ and the fact that $F_{Q}=0$ :
\\
-- {\em 1st proof, on $\Sigma$ using balayage} (recall that $d-2\leq s<d$): inequalities (\ref{Frostman1})-(\ref{Frostman2}) become
\begin{align*}
U^{\mu_{Q}}(x)-U^{\delta_{0}}(x) & \geq F_{Q},\quad\text{q.e. on }\Sigma, \\[5pt]
U^{\mu_{Q}}(x)-U^{\delta_{0}}(x) & \leq F_{Q}, \quad x\in S_{\mu_Q},
\end{align*}
which are satisfied when $\mu_{Q}=\hat\delta_{0}$, the balayage of $\delta_{0}$ on $\Sigma$ (see Section 3.2). Note that, by \cite[Theorem 3.22]{FZ} and assumption A2, we know that there is no mass loss when sweeping out $\delta_{0}$ onto $\Sigma$, thus $\|\hat\delta_{0}\|=1$, as required. Note also that we get $F_{Q}=0$ for the value of the Robin constant.
\\
-- {\em 2nd proof, using the Kelvin transform of center 0 and radius 1} : 
note first, that $|x||x^{*}|=1$ and for a general $c$, (\ref{F3})-(\ref{F4}) translate into
\begin{align*}
U^{\mu_{Q}^{*}}(x)-\frac{F_{Q}}{|x|^{s}} & \geq c,\quad\text{q.e. on }\Sigma^{*}, 
\\[5pt]
U^{\mu_{Q}^{*}}(x)-\frac{F_{Q}}{|x|^{s}} & \leq c, \quad x\in S_{\mu_{Q}^{*}}, 
\end{align*}
with the condition $U^{\mu_{Q}^{*}}(0)=1$. Assume $F_{Q}=0$. Then the above inequalities characterize, up to a constant, the unweighted equilibrium measure for the compact set $\Sigma^{*}$. Then, we know that $U^{\mu_{Q}^{*}}(x)=c$ q.e.\ on $\Sigma$ and, in particular, because of assumptions A1-A2, equality holds at 0, which induces
 $c=U^{\mu_{Q}^{*}}(0)=1$. In particular, we see that the unweighted energy problem on $\Sigma^{*}$ corresponds to the energy problem on $\Sigma$ with external field $-1/|x|^{s}$ and, as a consequence, we deduce again the existence of $\mu_{Q}$ on $\Sigma$ when $c=1$.
 \qed
\begin{remark}\label{rem-locat}
We just saw that the energy problem with $Q(x)=-1/|x|^{s}$ on $\Sigma$ corresponds to the unweighted energy problem on the compact set $\Sigma^{*}$. When $s=d-2$, it is known that the support of the equilibrium measure is included in the outer boundary of $\Sigma^{*}$, and thus the support of $\mu_{Q}$ is included in the boundary of $\Sigma$, which may be a bounded set.
\end{remark}

\subsection{Proof of Theorem \ref{thm:weakadmiss}}

%
From Lemma \ref{lem:balayage}, the weak balayage of $\nu$ is given by \eqref{nubalay}
\begin{equation*}
d \nu^w (x) = \frac{2^{\alpha}}{W(S^{d}) \omega_{d}} \left(\int \frac{|y_{d+1}|^{\alpha}}{|x - y|^{2d-s}}\,d \nu(y) \right)d x,\quad x\in \R^d.
\end{equation*}
Since $\nu (\R^{d+1}) = -1$, the signed equilibrium measure $\eta_Q$ exists and agrees with $-\nu^w$, i.e.
\begin{equation*}
d \eta_Q(x) = - \frac{2^{\alpha}}{W(S^{d}) \omega_{d}}\left(\int \frac{|y_{d+1}|^{\alpha}}{|x - y|^{2d-s}}\,d \nu(y) \right)\,d x\,,\quad x\in \R^d,
\end{equation*}
which implies that, for $x$ large, its density behaves like
$$
- \frac{2^{\alpha}}{W(S^{d}) \omega_{d} |x|^{2d-s}}\int |y_{d+1}|^{\alpha}\,d\nu(y).
$$
From condition \eqref{sufficRd}, this density is negative which guarantees the compactness of $S_{\eta_Q^+}$, and then Lemma \ref{lem:inclusion} ensures the compactness of $S_{\mu_Q}$.
\qed
\subsection{Proof of Theorems \ref{thm:singleattr} and \ref{thm:weakly}}
In this subsection we assume that $d-2\leq s<d$, therefore the weak balayage from Lemma \ref{lem:balayage} is regular balayage. The next lemma we study the balayage $Bal (\delta_y, B_{R})$ of $\delta_{y}$ onto the ball $B_{R}\subset\R^{d}$ of radius $R$. By symmetry, its density $Bal' (\delta_y, B_{R})(x)$ is a radial function of $x$. As we are about to see, it behaves like $(R^2 - |x|^2)^{-\alpha/2}$ near the boundary of $B_{R}$. Hence, it will be convenient to introduce the following notation,
\begin{align}
\Lambda_{R}(x) & = (R^2 - |x|^2)^{\alpha/2}Bal' (\delta_y, B_R)(x),\qquad |x|<R, \notag
\\[10pt]
\Lambda_{R}^{*} & = \lim_{|x|\to R_{-}}(R^2 - |x|^2)^{\alpha/2}Bal' (\delta_y, B_R)(x).
\label{Lambda*}
\end{align}
\begin{lemma}\label{lem:balR}
Let $y = (0; y_{d+1})$ with $y_{d+1} >0$. The following holds true:
\\
(i) The density $Bal' (\delta_y, B_{R})(x)$, $|x|<R$,
is given by
\begin{equation}\label{balayageBR}
Bal' (\delta_y, B_R)(x) = \frac{(2y_{d+1})^{\alpha}}{W(S^{d})\omega_{d}}
 \left(\frac{1}{(|x|^2 + y_{d+1}^2)^{d-s/2}}+
 \frac{\sin(\alpha\frac{\pi}{2})I(x)}{\pi (R^2 - |x|^2)^{\alpha/2}}\right),
\end{equation}
where 
\begin{equation*}
I(x)=\int_0^{+\infty}\frac{v^{\alpha/2}\,dv}{(v+R^2+y^2_{d+1})^{d-s/2} (v+R^2-|x|^2)}.
\end{equation*}
(ii) Let $m_R$ be the mass of $Bal (\delta_y, B_R)$. Then
\begin{equation}\label{massRy}
m_R = \frac{U^{\omega_R}(y)}{W(B_R)}\in(0,1).
\end{equation}
The mass $m_{R}$ is a non-decreasing function of $R$ and a decreasing function of $y_{d+1}$. Moreover
\begin{equation}\label{lim-mR}
\lim_{R\to 0}m_R = 0,\qquad \lim_{R\to\infty}m_R = 1.
\end{equation}
The second limit can be made more precise, namely, as $R\to\infty$,
\begin{equation}\label{lim-mR-inf}
m_R = 1-\frac{2}{\alpha B(s/2,\alpha/2)}
\left(\frac{y_{d+1}}{R}\right)^{\alpha}+o\left(\left(\frac{y_{d+1}}{R}\right)^{\alpha}\right)\quad\text{as }
R\to\infty.
\end{equation}
(iii) With the notation introduced in (\ref{Lambda*}), we have
\begin{equation}\label{beh-bal}
\Lambda_{R}^{*}=
\frac{y_{d+1}^{\alpha}K_{s,d}^{(1)}}{(R^2 + y_{d+1}^2)^{d/2}},\qquad
K_{s,d}^{(1)}=\frac{2^{\alpha}\sin(\alpha\frac{\pi}{2})B(d/2,\alpha/2)}
{\pi \omega_{d}W(S^{d})}.
\end{equation}
(iv) For $x\in B_{R}$, we have
\begin{multline}\label{ineq-Lamb}
\Lambda_{R}(x)-\Lambda_{R}^{*}
=\frac{(2y_{d+1})^{\alpha}(R^2 - |x|^2)}{W(S^{d})\omega_{d}}
\\[10pt]
\times
 \left(\frac{(R^2 - |x|^2)^{\alpha/2-1}}{(|x|^2 + y_{d+1}^2)^{d-s/2}}
-\frac{\sin(\alpha\frac{\pi}{2})}{\pi}
\int_{0}^{+\infty}\frac{v^{\alpha/2-1}dv}{(v+R^2+y^2_{d+1})^{d-s/2} (v+R^2-|x|^2)}\right)\geq0.
\end{multline}
\end{lemma}
\begin{proof}
{\bf (i)} From Lemma \ref{lem:balayage}, the \emph{superposition principle}, \cite[Eq.(4.5.6)]{Land}, and the expression for the Poisson kernel of a ball, \cite[Eq.(1.6.11), p.\ 121]{Land}, we derive that the density of the balayage of $\delta_y$ onto $B_R$ is given by
\begin{equation}\label{balBR}
Bal' (\delta_y, B_R)(x) =
\frac{(2y_{d+1})^{\alpha}}{W(S^{d})\,\omega_{d}}
\left(\frac{1}{(|x|^2 + y_{d+1}^2)^{d-s/2}}+\frac{\Gamma ({d}/{2})\sin (\alpha \pi/{2})A}
{\pi^{d/2+1}(R^2 - |x|^2)^{\alpha/2}}\right),
\end{equation}
where
\begin{equation*}
A =
\int_{|t|\geq R}\,\frac{(|t|^2 - R^2)^{\alpha/2}}{(|t|^2 + y_{d+1}^2)^{d-s/2}}\,
\frac{d t}{|x - t|^d}.
\end{equation*}
Using polar coordinates in $\R^d$, $A$ may be rewritten as
\begin{equation*}
2\pi\,\prod_{k=1}^{d-3}\,\int_0^{\pi}\,\sin^k \theta\,d\theta\,
\int_R^{+\infty}\,\int_0^{\pi}\,\frac{(\rho^2 - R^2)^{\alpha/2}\,\rho^{d-1}\,\sin^{d-2}\theta\,d\rho\,d\theta}{(\rho^2 + y_{d+1}^2)^{d-s/2}\,(\rho^2+|x|^2 - 2\rho |x| \cos \theta)^{d/2}}\,.
\end{equation*}
Setting $\rho = \rho_1 |x|$, we have
$$\int_0^{\pi}\,\frac{\sin^{d-2}\theta\,d\theta}{(\rho^2+|x|^2 - 2\rho |x| \cos \theta)^{d/2}} = \,\frac{1}{|x|^d}\,\int_0^{\pi}\,\frac{\sin^{d-2}\theta\,d\theta}{(\rho_1^2+1 - 2\rho_1 \cos \theta)^{d/2}}.$$
From \cite[p.\ 400]{Land}, the
above expression equals
$$
\frac{1}{\rho^{d-2} (\rho^2-|x|^2)}\,\int_0^{\pi}\,\sin^{d-2} \theta d\theta,$$
which allows us to write $A$ as 
\begin{align*}
A & = 2\pi\,\prod_{k=1}^{d-2}\,\int_0^{\pi}\,\sin^k \theta\,d\theta\,
\int_R^{+\infty}\frac{(\rho^2 - R^2)^{\alpha/2}\,\rho\,d\rho}{(\rho^2 + y_{d+1}^2)^{d-s/2}\,(\rho^2-|x|^2)}
\\[10pt]
& = \frac{\pi^{d/2}}{\Gamma(d/2)}\,\int_0^{+\infty}\,\frac{v^{\alpha/2}\,dv}{(v+R^2+y^2_{d+1})^{d-s/2} (v+R^2-|x|^2)}.
\end{align*}
This, along with \eqref{balBR}, complete the proof of (\ref{balayageBR}).
\\
{\bf (ii)} Formula \eqref{massRy} follows from \cite[Eq.(4.5.6')]{Land}.
From the second Frostman inequality and the domination principle, $U^{\omega_R}(y) \leq W(B_R)$ and thus, $m_R \leq 1$. Since the mass of a measure can only decrease when performing a balayage, and since, for $R<R'$, the balayage onto $B_{R}$ can be obtained by the balayage onto $B_{R'}$ followed by a balayage onto $B_{R}$, we derive that $m_{R}$ increases with $R$. From \eqref{massRy} it is immediate that $m_R$ is a decreasing function of $y_{d+1}$.

When $R\to0$, $U^{\omega_{R}}(y)\to|y|^{-s}$ while $W(B_{R})\to\infty$ which implies the first limit in (\ref{lim-mR}).
The behavior (\ref{lim-mR-inf}) of $m_{R}$ at infinity is derived from \eqref{massRy}, (\ref{ballenergy}), and an expansion of (\ref{U-ball}) when $R\to\infty$.
\\
{\bf (iii)} The behavior of the density of the balayage near the boundary of $B_{R}$ simply follows from
(\ref{balayageBR}) and the easily checked equality
$$
I(R)=(R^{2}+y_{d+1}^{2})^{-d/2}B(d/2,\alpha/2).
$$
{\bf (iv)} The expression (\ref{ineq-Lamb}) for $\Lambda_{R}(x)-\Lambda_{R}^{*}$ follows from the definition of $\Lambda_{R}(x)$, $\Lambda_{R}^{*}$, and the expression (\ref{balayageBR}) for the balayage. The inequality 
is equivalent to
\begin{align*}
\frac{(R^2 - |x|^2)^{\alpha/2-1}}{(|x|^2 + y_{d+1}^2)^{d-s/2}}
& \geq
\frac1\pi\sin(\alpha\frac{\pi}{2})
\int_0^{+\infty}\,\frac{v^{\alpha/2-1}\,dv}{(v+R^2+y^2_{d+1})^{d-s/2}(v+R^{2}-|x|^{2})}.
\end{align*}
But the last integral is smaller than
\begin{align*}
\frac{1}{(R^2+y^2_{d+1})^{d-s/2}}
\int_0^{+\infty}\,\frac{v^{\alpha/2-1}\,dv}{(v+R^{2}-|x|^{2})}
& =
\frac{(R^{2}-|x|^{2})^{\alpha/2-1}}{(R^2+y^2_{d+1})^{d-s/2}}
\int_0^{+\infty}\,\frac{v^{\alpha/2-1}\,dv}{(v+1)}
\\
& = \frac{(R^{2}-|x|^{2})^{\alpha/2-1}}{(R^2+y^2_{d+1})^{d-s/2}}
\frac{\pi}{\sin(\alpha\frac{\pi}{2})},
\end{align*}
which proves the inequality
since $(R^2+y^2_{d+1})^{d-s/2}
\geq(|x|^2+y^2_{d+1})^{d-s/2}$.
\end{proof}
\noindent
\textbf{Proof of Theorem \ref{thm:singleattr}}
{\bf (i)}
The proof is based on the signed equilibrium measures,
as explained in Section 3.2. Since the support $S_{\mu_Q}$ of the equilibrium measure $\mu_Q$ is a compact set, it is a subset of a ball $B_R$, for $R$ large enough. Pick such
an $R>0$. The signed equilibrium measure $\eta_{Q,R}$ 
of the ball $B_{R}$ exists and can be expressed in terms of the balayage $Bal (\delta_y,B_R)$ and the equilibrium measure $\omega_{R}$, 
\begin{equation}\label{signedsingleat}
\eta_{Q,R} = - \gamma Bal (\delta_y,B_R) + (1 + \gamma m_R)\omega_R.
\end{equation}
Indeed, the right-hand side satisfies (\ref{defsigned}) and has total mass 1.
In view of (\ref{equilball}), (\ref{beh-bal}) and (\ref{signedsingleat}), the density of $\eta_{Q,R}$ near the sphere $|x| = R$ satisfies
\begin{align*}
H(R) & := \lim_{|x|\rightarrow R}(R^2 - |x|^2)^{\alpha/2}\eta'_{Q,R}(x) =
-\gamma\Lambda_{R}^{*}+(1+\gamma m_{R})c_{R}
\\
& =-\frac{\gamma y_{d+1}^{\alpha}K_{s,d}^{(1)}}{(R^2 + y_{d+1}^2)^{d/2}}
+\frac{\Gamma (1+s/2)}{\pi^{d/2} \Gamma (1-\alpha/2)}\frac{(1+\gamma m_{R})}{R^{s}}.
\end{align*}
Observe that
$$
\lim_{R\to0}H(R)=+\infty,\qquad\lim_{R\to\infty}H(R)=0_{-},
$$
where, for the first limit, we note that $m_{R}\to0$ as $R\to0$, and, for the second limit, we note that, as $R\to\infty$, the second term is dominant (since $s<d$) and that $1+\gamma m_{R}\to1+\gamma<0$. Hence, by continuity, there exists values of $R$ for which $H(R)=0$. We denote by $R_{0}$ the largest such value (which exists).

We next check that $\eta_{Q,R_0}$ is nonnegative on the closed ball $B_{R_0}$.
Indeed, from (\ref{signedsingleat}), and the definition of $R_{0}$, we have
\begin{equation}\label{Eq-R0}
-\gamma\Lambda_{R_{0}}^{*}+(1+\gamma m_{R_{0}})c_{R_{0}}=0.
\end{equation}
Hence, for $|x|\leq R_{0}$,
$$
(R^2 - |x|^2)^{\alpha/2}\eta'_{Q,R}(x)
=-\gamma\Lambda_{R_{0}}(x)+(1+\gamma m_{R_{0}})c_{R_{0}}=
-\gamma(\Lambda_{R_{0}}(x)-\Lambda_{R_{0}}^{*})\geq0,
$$
where, for the last inequality, we use (\ref{ineq-Lamb}).

Now, for $R>R_{0}$, we know that $\eta'_{Q,R}(R)<0$. Hence by iii) of Lemma \ref{lem:inclusion}, $S_{\mu_Q}$ is included in $B_{R_{0}}$, and by ii) of that same lemma, $\mu_{Q}=\eta_{Q,R_{0}}$.
In particular, the support of $\mu_{Q}$ is the ball of radius $R_{0}$.

{\bf (ii)} By combining (\ref{signedsingleat}) with the explicit expressions (\ref{equilball}) and (\ref{balayageBR}) of $\omega_{R_{0}}$ and $Bal(\delta_{y},B_{R_{0}})$, together with (\ref{Eq-R0}), we may derive that
$$
\mu'_Q(x) = -\gamma \frac{(2y_{d+1})^{\alpha}}{W(S^{d})\omega_{d}}
\left(\frac{1}{(|x|^2 + y_{d+1}^2)^{d-{s}/{2}}}
- \frac{\sin(\alpha\pi/2)(I(x)-I(R_{0}))}{\pi(R_{0}^{2}-|x|^{2})^{\alpha/2}}
\right).
$$
Performing the change of variable $v = (R_0^2-|x|^2) u $ in $I(x)$ and $I(R_{0})$
leads to \eqref{muQ}. Finally, when $|x|=R_{0}$, the second factor of \eqref{muQ} reduces to
$$
\frac{1}{(R_{0}^2 + y_{d+1}^2)^{d-{s}/{2}}} \left(1
- \frac{\sin(\alpha\pi/2)}{\pi}\int_{0}^{+\infty}(u+1)^{-1}u^{\alpha/2-1}du\right)=0,
$$
which shows the vanishing of $\mu_{Q}'$ on the boundary of the ball.

{\bf (iii)} Recall that the radius $R_{0}$ of the ball $S_{\mu_Q}$ is a solution of (\ref{Eq-R0}), that is
$$
-\gamma\Lambda_{R_{0}}^{*}+(1+\gamma m_{R_{0}})c_{R_{0}}=0.
$$
Plugging the explicit expressions of $\Lambda_{R_{0}}^{*}$, $m_{R_{0}}$ and $c_{R_{0}}$ given respectively by (\ref{beh-bal}), (\ref{massRy}), and (\ref{equilball}), and making also use of (\ref{U-ball}) and (\ref{ballenergy}), one may check, after some computations, that the following equation is obtained,
\begin{equation}\label{Eq-final}
z^{s/2}\left({}_2F_{1}\left(\frac{s}{2},\frac{d}{2},1+\frac{s}{2},-z\right)
-(1+z)^{-d/2}\right)=-\frac{\Gamma(\alpha/2)\Gamma(1+s/2)}{\gamma\Gamma(d/2)},
\end{equation}
where $z=(R_{0}/y_{d+1})^{2}$. The left-hand side, that we denote by $G(z)$, is equal to
\begin{align*}
G(z) & =z^{s/2}\left({}_2F_{1}\left(\frac{s}{2},\frac{d}{2},1+\frac{s}{2},-z\right)
-{}_2F_{1}\left(1+\frac{s}{2},\frac{d}{2},1+\frac{s}{2},-z\right)\right) 
\\[10pt]
& = \frac{\Gamma(1+s/2)}{\Gamma(d/2)\Gamma(1-\alpha/2)}z^{s/2}
\int_{0}^{1}x^{d/2-1}(1-x)^{-\alpha/2}(1+zx)^{-s/2-1}(zx)dx 
\\[10pt]
& = \frac{d}{s+2}z^{s/2+1} {}_2F_{1}\left(1+\frac{s}{2},1+\frac{d}{2},2+\frac{s}{2},-z\right).
\end{align*}
This gives (\ref{sol}). Finally, we show that a solution to equation (\ref{Eq-final}) exists and is unique. First, we notice that $G(z)$
 is a non-decreasing function of $z\in[0,\infty)$. Indeed,
\begin{align}
G(z)
 = \frac{\Gamma(1+s/2)}{\Gamma(d/2)\Gamma(1-\alpha/2)}
\int_{0}^{1}x^{d/2}(1-x)^{-\alpha/2}(x+1/z)^{-s/2-1}dx, \label{integ-G}
\end{align}
and it is clear that, for any $x\in(0,1)$, $(x+1/z)^{-s/2-1}$ is a non-decreasing function of $z$.
Second, $G(z)$ satisfies $\lim_{z\to0} G(z)= 0$ and
$$
\lim_{z\to\infty} G(z)=
\frac{\Gamma(\alpha/2)\Gamma(1+s/2)}{\Gamma(d/2)},
$$
where the above limit follows, e.g., from (\ref{integ-G}). Third the right hand-side of (\ref{Eq-final}) is positive and always smaller than the above limit because $\gamma<-1$. Actually it tends to that  limit as $\gamma$ tends to $-1$ from the left.
\qed
\\[10pt]
\textbf{Proof of Theorem \ref{thm:weakly}}
The density (\ref{signedattrep}) of the signed equilibrium measure $\eta_{Q,\R^{d}}$ vanishes when
$$
(1+\gamma)y_{1,d+1}^{\alpha}(|x|^{2}+ y_{2,d+1}^{2})^{d-s/2}-\gamma y_{2,d+1}^{\alpha}
(|x|^{2} + y_{1,d+1}^{2})^{d-s/2}=0.
$$
Solving for $|x|^{2}$, we get
$$
|x|^{2}=\frac{\rho y_{1,d+1}^{2}-y_{2,d+1}^{2}}{1-\rho}=R^{2},
$$
where $\rho$ and $R$ are defined in (\ref{def-g-R}). Hence, $\eta_{Q,\R^{d}}$ is a positive measure if and only if
$$
\frac{\rho y_{1,d+1}^{2}-y_{2,d+1}^{2}}{1-\rho}\leq0\qquad\text{and}\qquad
(1+\gamma)y_{1,d+1}^{\alpha}-\gamma y_{2,d+1}^{\alpha}\geq0,
$$
where the second inequality expresses the fact that the density of $\eta_{Q,\R^{d}}$ is positive at infinity. It is then straightforward to see that the above inequalities are equivalent to
$y_{2,d+1}/y_{1,d+1}\in[g^{1/d},g^{-1/\alpha}]$. Finally, if the first inequality is not met, $\eta_{Q,\R^{d}}$ is a true signed measure and one checks that
\begin{align*}
& \eta_{Q,\R^{d}}(x)<0,~|x|\in[0,R),\quad\eta_{Q,\R^{d}}(x)\geq0,~|x|\in[R,\infty),\quad\text{when }y_{2,d+1}/y_{1,d+1}\in[0,g^{1/d}],
\\
& \eta_{Q,\R^{d}}(x)\geq0,~|x|\in[0,R],\quad\eta_{Q,\R^{d}}(x)<0,~|x|\in(R,\infty),\quad\text{when }y_{2,d+1}/y_{1,d+1}\in[g^{-1/\alpha},\infty).
\end{align*}
Together with (i) of Lemma \ref{lem:inclusion}, it implies the statements (i) and (iii) and finishes the proof of the theorem.
\qed
\section*{Acknowledgements}
We thank N.V. Zorii for useful comments on this paper.
The research of the first author was supported in part by NSF grant DMS-1936543.
The research of the second author was partially supported by the Spanish Ministerio de Ciencia,
Innovaci\'{o}n y Universidades, under grant MTM2015-71352-P, by Spanish Ministerio de Ciencia e Innovaci\'{o}n, under grant PID2021-123367NB-100, and by the PFW Scholar-in-Residence program.

\obeylines
\texttt{P.~Dragnev (dragnevp@pfw.edu)
Department of Mathematical Sciences, Purdue University Fort Wayne, 
Ft.\ Wayne, IN 46805, USA.
\medskip
R.~Orive (rorive@ull.es)
Departmento de An\'{a}lisis Matem\'{a}tico, Universidad de La Laguna, 
38200 La Laguna (Tenerife), SPAIN.
\medskip
E.~B.~Saff (edward.b.saff@vanderbilt.edu)
Center for Constructive Approximation, Vanderbilt University, 
Department of Mathematics, Nashville, TN  37240, USA.
\medskip
F.~Wielonsky (franck.wielonsky@univ-amu.fr)
Laboratoire I2M, UMR CNRS 7373, Universit\'e Aix-Marseille, 
39 Rue Joliot Curie, F-13453 Marseille Cedex 20, FRANCE.
}


\begin{thebibliography}{10}

\bibitem{Abramowitz}
M. Abramowitz, I. A. Stegun, \newblock Handbook of mathematical functions with formulas, graphs, and mathematical tables.
\newblock Dover, 1970.
\bibitem{B} R.G. Bartle, A Modern Theory of Integration. Graduate Studies in Mathematics, 32. American Mathematical Society, Providence, RI, 2001.

\bibitem{BDO}
D. Benko, P. Dragnev and R. Orive, \newblock On point-mass Riesz external fields on the real axis.
\newblock {\em J. Math. Anal. Appl.} 491 (2020), 124--299.

\bibitem{BDT2012}
D. Benko, P. D. Dragnev and V. Totik, \newblock Convexity of harmonic densities.
\newblock {\em Rev. Mat. Iberoam.} 28 (2012), no. 4, 947--960.

\bibitem{Bilog}
M. Bilogliadov, \newblock Minimum Riesz energy problem on the hyperdisk.
\newblock J. Math. Phys. 59 (2018), 013301.


\bibitem{BLW}
T. Bloom, N. Levenberg and F. Wielonsky,
\newblock Logarithmic potential theory and large deviation.
\newblock {\em Comp. Meth. Funct.} Th. 15(4), 555--594 (2015).

\bibitem{BLW2}
T. Bloom, N. Levenberg and F. Wielonsky,
\newblock A large deviation principle for weighted Riesz interactions.
\newblock {\em Constr. Approx.} 47 (2018), 119-140.

\bibitem{BHS} S. V. Borodachov, D. P. Hardin and E. B. Saff, Discrete Energy on Rectifiable Sets, Springer Monographs in Mathematics, Springer, 2019.

\bibitem{BDS2009}
J. S. Brauchart, P. D. Dragnev and E. B. Saff, \newblock Riesz extremal measures on the sphere for axis-supported external fields.
\newblock {\em J. Math. Anal. Appl.} 356 (2009) 769--792.


\bibitem{Br} M. Brelot,
Sur le r\^ole du point \`a
l'infini dans la th\'eorie des fonctions harmoniques.
{\em Ann. Sci. Ecole Norm. Sup.} 61, (1944). 301-332.

\bibitem{Br2} M. Brelot,
Lectures in Potential Theory, Tata Institute, Bombay, 1960.

\bibitem{Buc}
C. Bucur, Some observations on the Green function for the ball in the fractional Laplace framework, {\em Comm. Pure Appl. Anal.} 15 (2016), 657-699.

\bibitem{chafai}
D. Chafai, N. Gozlan and P. A. Zitt,
\newblock First--order global asymptotics for confined particles with singular pair repulsion.
\newblock {\em Ann. Appl. Probab.} 24 (2014), no. 6, 2371--2413.



\bibitem{DLVP}
Ch. de La Vall\'{e}e--Poussin,
\newblock Potentiel et probl\`eme g\'{e}n\'{e}ralis\'{e} de Dirichlet,
\newblock {\em Math. Gazette, London} 22 (1938), 17--36.


\bibitem{D} J.L. Doob, Classical Potential Theory and Its Probabilistic Counterpart, Springer, Berlin (1984).

\bibitem{DS} P. D. Dragnev and E. B. Saff, Riesz spherical
potentials with external fields and minimal energy points separation, {\em Potential Anal.} {\bf 26} (2007), 139--162.

\bibitem{Fu} B. Fuglede, On the theory of potentials in locally compact spaces, {\em Acta Math.} 103(3-4), (1960), 139-215.

\bibitem{F} B. Fuglede, Some properties of the Riesz charge associated with a $\delta$-subharmonic function. {\em Potential Anal.} 1, (1992), 355--371.

\bibitem{FZ} B. Fuglede and N. Zorii, Green kernels associated with Riesz kernels. {\em Ann. Acad. Sci. Fenn. Math.} 43,  (2018), 121--145.

\bibitem{Hardy} A. Hardy,
\newblock A note on large deviations for 2D Coulomb gas with weakly confining potential,
\newblock {\em Electron. Commun. Probab.} 17 (2012), no. 19, 1--12.

\bibitem{Hardy-Kuijlaars}
A. Hardy, A. B. J. Kuijlaars,
\newblock Weakly admissible vector equilibrium problems,
\newblock J. Approx. Th. 164(6), 854--868 (2012).

\bibitem{J} K. Janssen, On the Choquet charge of $\delta$-superharmonic functions, {\em Potential Anal.} {\bf 12}, (2000), 211-220.

\bibitem{KD}
A. B. J. Kuijlaars and P. Dragnev,
\newblock Equilibrium problems associated with fast decreasing polynomials.
\newblock {\em Proc. Am. Math. Soc.} 127, 1065--1074 (1999).

\bibitem{KM} T. Kurokawa and Y. Mizuta, On the order at infinity of Riesz potentials. {\em Hiroshima Math. J.} 9 (1979), no. 2, 533-545.

\bibitem{leble}
T. Lebl\'{e} and S. Serfaty,
\newblock Large deviation principle for empirical fields of log and Riesz gases.
\newblock {\em Invent. Math.} 210 (2017), no. 3, 645--757.

\bibitem{Land}
N. S. Landkof,
\newblock Foundations of Modern Potential Theory, Grundlehren der mathematischen
Wissenschaften 180.
\newblock {\em Springer-Verlag, New York-Heidelberg}, 1972.

\bibitem{abey}
A. L\'{o}pez Garc\'{i}a,
\newblock Greedy energy points with external fields.
\newblock {\em Contemporary Mathematics} 507 (2010), 189--207.





\bibitem{M} Y. Mizuta,
Potential Theory in Euclidean Spaces.
Gakuto International Series. Mathematical Sciences and Applications, 6. Gakkotosho Co., Ltd., Tokyo, 1996.


\bibitem{O} M. Ohtsuka,
On potentials in locally compact spaces.
{\em J. Sci. Hiroshima Univ. Ser. A-I Math.} 25 (1961), 135-352.

\bibitem{OSW}
R. Orive, J. S\'{a}nchez Lara and F. Wielonsky,
\newblock Equilibrium problems in weakly admissible external fields created by pointwise charges,
\newblock {\em J. Approx. Theory} 244 (2019), 71--100.


\bibitem{Ra} T. Ransford, \emph{Potential Theory in the Complex
Plane}, Cambridge University Press, Cambridge, 1995.

\bibitem{R1} M. Riesz, Int\'egrales de Riemann-Liouville et potentiels,
{\em Acta Szeged}, 9 (1938), 1-42.

\bibitem{R2} M. Riesz, L'int\'egrale de Riemann-Liouville et le probl\`eme de Cauchy. {\em Acta Math.} 81 (1949), 1-223.

\bibitem{R} W. Rudin, Principles of Mathematical Analysis, McGraw-Hill, 1976.

\bibitem{ST}
E.~B. Saff and V.~Totik,
\newblock Logarithmic Potentials with External Fields, volume 316 of
  Grundlehren der Mathematischen Wissenschaften.
Springer-Verlag, Berlin, 1997.

\bibitem{Simeonov} P. Simeonov,
\newblock A weighted energy problem for a class of admissible weights,
\newblock {\em Houston J. Math.} 31(4), 1245--1260 (2005).

\bibitem{Sodin} M. Sodin,
Hahn decomposition for the Riesz charge of $\delta$-subharmonic functions. {\em Math. Scand.} 83 (1998), no. 2, 277-282.

\bibitem{Z0} N. Zorii, On the solvability of the Gauss variational problem, {\em Comput. Meth. Funct.} Th. 2, (2002), 427-448.

\bibitem{Z1} N.V. Zorii, Equilibrium potentials with external fields. {\em Ukrainian Math. J.} 55 (9) (2003), 1423-1444.

\bibitem{Z2} N. Zorii, Equilibrium problems for infinite dimensional vector potentials with external fields.
{\em Potential Anal.} 38 (2) (2013), 397-432.


\bibitem{Z3} N. Zorii, Harmonic measure, equilibrium measure, and thinness at infinity in the theory of Riesz potentials, {\em Potential Anal.} (2021) DOI 10.1007/s11118-021-09923-2.

\end{thebibliography}
\end{document}